\documentclass[oneside]{amsart}
\usepackage[utf8]{inputenc}
\usepackage{cite}
\usepackage{hyperref}

\usepackage[a4paper, total={6in, 8.5in}]{geometry}
\usepackage{graphicx}
\usepackage{amssymb}
\usepackage{tikz-cd}
\usepackage{mathrsfs} 
\usepackage[all]{xy}
\usepackage{amsthm}
\usepackage{tabularx}
\usepackage{nicefrac}
\usepackage{amsmath}
\usepackage{tikz}
\usepackage{pinlabel}
\usepackage[makeroom]{cancel}
\usepackage{array}
\usepackage{hyperref}

\newtheorem{theorem}{Theorem}[section]
\newtheorem{prop}[theorem]{Proposition}
\newtheorem{lemma}[theorem]{Lemma}

\newtheorem{defn}[theorem]{Definition}

\newtheorem{corollary}[theorem]{Corollary}
\theoremstyle{remark}
\newtheorem{remark}[theorem]{Remark}

\newcommand{\R}{\mathbb{R}}
\newcommand{\C}{\mathbb{C}}
\newcommand{\Z}{\mathbb{Z}}

\newcommand{\Q}{\mathbb{Q}}
\newcommand{\cp}{\mathbb{C}\mathbb{P}^2}
\newcommand{\cpbar}{\overline{\mathbb{C}\mathbb{P}}^2}
\newcommand{\xin}{\bar{x}}
\newcommand{\yin}{\bar{y}}
\newcommand{\zin}{\bar{z}}

\title[Definite lattices bounded by integer surgeries on knots with small slice genus]{On definite lattices bounded by integer surgeries along knots with slice genus at most 2}
\author{Marco Golla}
\email{marco.golla@univ-nantes.fr}
\address{CNRS, Laboratoire de Math\'ematiques Jean Leray, Nantes, France}
\thanks{MG acknowledges support from CNRS though a ``Jeunes chercheurs et jeunes chercheuses'' grant, and hospitality from the Simons Center for Geometry and Physics.}

\author{Christopher Scaduto}
\email{cscaduto@scgp.stonybrook.edu}
\address{Simons Center for Geometry and Physics, Stony Brook, NY, U.S.A.}
\thanks{CS was supported by NSF grant DMS-1503100.}

\date{}

\begin{document}

\maketitle 

\begin{abstract}
We classify the positive definite intersection forms that arise from smooth 4-manifolds with torsion-free homology bounded by positive integer surgeries on the right-handed trefoil.
A similar, slightly less complete classification is given for the $(2,5)$-torus knot, and analogous results are obtained for integer surgeries on knots of slice genus at most two.
The proofs use input from Yang--Mills instanton gauge theory, Heegaard Floer correction terms, and the topology of singular complex plane curves.
\end{abstract}


\section{Introduction}\label{sec:intro}

A {\emph{lattice}} is a free abelian group of finite rank equipped with a symmetric bilinear form. A lattice is {\emph{integral}} if the pairing has values in $\Z$. Let $Y$ be a rational homology 3-sphere, that is, a closed, oriented and connected 3-manifold with $b_1(Y)=0$.
Let $X$ be a smooth, compact and oriented 4-manifold bounded by $Y$ with no torsion in its homology.
The intersection of 2-dimensional cycles in $X$ induces the structure of an integral lattice on $H_2(X;\Z)$, which we denote by $L_X$, the {\emph{intersection form}} of $X$. This situation is summarized in:
\begin{defn}
	We say a rational homology $3$-sphere $Y$ {\emph{bounds}} a lattice $L$ and that $L$ {\emph{fills}} $Y$ if $L=L_X$ for a smooth, compact, oriented $4$-manifold $X$ with no torsion in its homology and $\partial X =Y$.\label{def:filling}
\end{defn}

What are the possible {\emph{definite}} fillings of a fixed rational homology 3-sphere $Y$? When $Y$ is the 3-sphere, Donaldson's Theorem \cite{d-connections} says any such filling is diagonal. When $Y$ is the Poincar\'{e} sphere, Fr\o yshov's Theorem \cite{froyshov-thesis} says the only such non-diagonal fillings are $-E_8\oplus \langle - 1\rangle^k$. These theorems are reproved in \cite{scaduto-forms}, where further results in this direction are provided, all for {\emph{integer}} homology 3-spheres. The purpose of the current article is to provide analogous results for when $Y$ has non-trivial first homology, and in particular when $Y$ is a positive, integral surgery along a knot. For a knot $K$ in the 3-sphere, we write $S^3_n(K)$ for the 3-manifold obtained as $n$-surgery along $K$.

\begin{theorem}\label{thm:trefoil}
{\emph{(i)}} Let $K\subset S^3$ be a knot of slice genus at most $1$, and $n\in \mathbb{Z}_{> 0}$. If a positive definite lattice fills $S_n^3(K)$, then it is isomorphic to one of the following, for some $k\geqslant 0$:
    \begin{align*}
        n=1: \qquad& \;\, \phantom{\langle 1 \rangle \oplus} \langle 1 \rangle^k  \qquad \text{or} \qquad E_8\oplus \langle 1 \rangle^k \qquad\\
        n=2: \qquad& \langle 2 \rangle \oplus \langle 1 \rangle^k  \qquad \text{or} \qquad E_7\oplus \langle 1 \rangle^k \qquad\\
        n=3: \qquad& \langle 3 \rangle \oplus \langle 1 \rangle^k  \qquad \text{or} \qquad E_6\oplus \langle 1 \rangle^k \qquad\\
        n=4:\qquad & \langle 4 \rangle \oplus \langle 1 \rangle^k  \qquad \text{or} \qquad D_5\oplus \langle 1 \rangle^k \qquad\\
        n=5: \qquad& \langle 5 \rangle \oplus \langle 1 \rangle^k  \qquad \text{or} \qquad A_4\oplus \langle 1 \rangle^k \qquad\\
        n = 6: \qquad& \langle 6 \rangle \oplus \langle 1 \rangle^k   \qquad \text{or} \qquad A_1\oplus A_2\oplus \langle 1 \rangle^k \qquad\\
        n = 7: \qquad& \langle 7 \rangle \oplus \langle 1 \rangle^k   \qquad \text{or} \qquad \Lambda(2,4)\oplus \langle 1 \rangle^k \qquad\\
        n\geqslant 8: \qquad& \langle n \rangle \oplus \langle 1 \rangle^k 
    \end{align*}
    {\emph{(ii)}} Furthermore, when the knot is the right-handed trefoil, all of these possibilities are realized, except for the trivial lattice.
\end{theorem}


Here we write $A_n$ ($n\geqslant 1$), $D_n$ ($n\geqslant 4$), and $E_6$, $E_7$, $E_8$ for the positive definite root lattices corresponding to Dynkin diagrams of type ADE, and $\langle n \rangle$ for the rank-1 lattice generated by a vector $v$ with $v\cdot v = v^2 = n$. Note that $A_1=\langle 2\rangle$. We write $\Lambda(a_1,\ldots,a_n)$ for the linearly plumbed lattice with basis $e_1,\ldots,e_n$ whose only nonzero pairings are $e_i\cdot e_{i+1}=-1$ for $1\leqslant i \leqslant n-1$ and $e_i\cdot e_i=a_i$ for $1\leqslant i \leqslant n$. When $a_1=\cdots = a_n=2$ we recover $A_n$.

The lattices in the right-hand column of Theorem~\ref{thm:trefoil} are obtained by starting with a root diagram for $E_8$ and successively plucking off the last root from the end of the long leg, as in Figure~\ref{f:Tns}. The exception to this procedure is passing from $A_1\oplus A_2$ to $\Lambda(2,4)$.

\begin{figure}[h]
\begin{center}
\begin{tikzpicture}[scale=.5]
	\draw (0,0) -- (0,6);
	\draw (-1,4) -- (0,4);
	
	\draw[fill=black] (0,0) circle(.1);
	\draw[fill=black] (0,1) circle(.1);
	\draw[fill=black] (0,2) circle(.1);
	\draw[fill=black] (0,3) circle(.1);
	\draw[fill=black] (0,4) circle(.1);
	\draw[fill=black] (0,5) circle(.1);
	\draw[fill=black] (0,6) circle(.1);
	\draw[fill=black] (-1,4) circle(.1);

    \draw (3,1) -- (3,6);
	\draw (2,4) -- (3,4);
	
	\draw[fill=black] (3,1) circle(.1);
	\draw[fill=black] (3,2) circle(.1);
	\draw[fill=black] (3,3) circle(.1);
	\draw[fill=black] (3,4) circle(.1);
	\draw[fill=black] (3,5) circle(.1);
	\draw[fill=black] (3,6) circle(.1);
	\draw[fill=black] (2,4) circle(.1);
	
	\draw (6,2) -- (6,6);
	\draw (5,4) -- (6,4);
	
	\draw[fill=black] (6,2) circle(.1);
	\draw[fill=black] (6,3) circle(.1);
	\draw[fill=black] (6,4) circle(.1);
	\draw[fill=black] (6,5) circle(.1);
	\draw[fill=black] (6,6) circle(.1);
	\draw[fill=black] (5,4) circle(.1);
	
	\draw (9,3) -- (9,6);
	\draw (8,4) -- (9,4);
	
	\draw[fill=black] (9,3) circle(.1);
	\draw[fill=black] (9,4) circle(.1);
	\draw[fill=black] (9,5) circle(.1);
	\draw[fill=black] (9,6) circle(.1);
	\draw[fill=black] (8,4) circle(.1);
	
	\draw (12,4) -- (12,6);
	\draw (11,4) -- (12,4);
	
	\draw[fill=black] (12,4) circle(.1);
	\draw[fill=black] (12,5) circle(.1);
	\draw[fill=black] (12,6) circle(.1);
	\draw[fill=black] (11,4) circle(.1);
	
	\draw (15,5) -- (15,6);
	
	\draw[fill=black] (15,5) circle(.1);
	\draw[fill=black] (15,6) circle(.1);
	\draw[fill=black] (14,4) circle(.1);

	\draw (18,6) -- (18,5);

	\draw[fill=black] (18,6) circle(.1);
	\draw[fill=black] (18,5) circle(.1);
	
	\node at (0,7) {$\mathscr{T}_1$};
	\node at (0,-1) {$E_8$};
	\node at (3,7) {$\mathscr{T}_2$};
	\node at (3,0) {$E_7$};
	\node at (6,7) {$\mathscr{T}_3$};
	\node at (6,1) {$E_6$};
	\node at (9,7) {$\mathscr{T}_4$};
	\node at (9,2) {$D_5$};
	\node at (12,7) {$\mathscr{T}_5$};
	\node at (12,3) {$A_4$};
	\node at (15,7) {$\mathscr{T}_6$};
	\node at (15,3) {$A_1\oplus A_2$};
	\node at (18,7) {$\mathscr{T}_7$};
	\node at (17.5,5) {$4$};
	\node at (18,4) {$\Lambda(2,4)$};
\end{tikzpicture}
\end{center}
\caption{The trefoil lattices $\mathscr{T}_n$.}\label{f:Tns}
\end{figure}
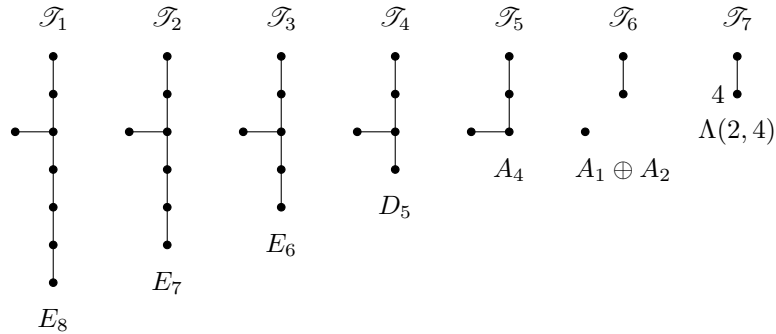

A description of this process in terms of successively taking orthogonal complements of vectors of squared norm $n(n+1)$ is given in Section~\ref{sec:lattices}. We next provide a partial analogue of Theorem~\ref{thm:trefoil} for knots of slice genus at most 2, building on the main result of~\cite{scaduto-forms}.

\begin{theorem}\label{thm:2,5}
{\emph{(i)}} Let $K\subset S^3$ be a knot of slice genus at most $2$, and $n\in\mathbb{Z}_{> 0}$. If a positive definite lattice fills $S_n^3(K)$, then it is either isomorphic to a lattice in Theorem~\ref{thm:trefoil} (i), or to one of the following lattices, for some $k\geqslant 0$:
    \[
        \mathscr{C}_{n}\oplus \langle 1\rangle^k \;\; (1\leqslant n \leqslant 11),   \qquad E_8 \oplus \langle 2 \rangle \oplus \langle 1\rangle \oplus \langle 1 \rangle^k  \;\; (n=2), \qquad E_8\oplus \langle 3\rangle \oplus \langle 1 \rangle^k \;\;(n=3),
    \]
or possibly $\Gamma_{12}\oplus \langle n \rangle \oplus \langle 1 \rangle^k$ if $n\in\{2,3\}$. Here $\mathscr{C}_n$ for $1\leqslant n\leqslant 11$ are the lattices defined by the plumbing diagrams in Figure~\ref{f:Cns} below.

    \noindent {\emph{(ii)}} Furthermore, when the knot is the right-handed cinquefoil, i.e. the positive $(2,5)$-torus knot, all of these possibilities are realized, except possibly $\Gamma_{12}\oplus \langle n \rangle \oplus \langle 1 \rangle^k$ where $n\in\{2,3\}$.
    \end{theorem}
    
\begin{figure}[h!]
        \begin{center}
\begin{tikzpicture}[scale=.4]
	\draw (0,-4) -- (0,6);
	\draw (-1,4) -- (0,4);
	
	\draw[fill=black] (0,0) circle(.1);
	\draw[fill=black] (0,1) circle(.1);
	\draw[fill=black] (0,2) circle(.1);
	\draw[fill=black] (0,3) circle(.1);
	\draw[fill=black] (0,4) circle(.1);
	\draw[fill=black] (0,5) circle(.1);
	\draw[fill=black] (0,6) circle(.1);
	\draw[fill=black] (0,-1) circle(.1);
	\draw[fill=black] (0,-2) circle(.1);
	\draw[fill=black] (0,-3) circle(.1);
	\draw[fill=black] (0,-4) circle(.1);
	\draw[fill=black] (-1,4) circle(.1);
	\node at (-.55,6) {$3$};

    \draw (3,-3) -- (3,6);
	\draw (2,4) -- (3,4);
	
	\draw[fill=black] (3,1) circle(.1);
	\draw[fill=black] (3,2) circle(.1);
	\draw[fill=black] (3,3) circle(.1);
	\draw[fill=black] (3,4) circle(.1);
	\draw[fill=black] (3,5) circle(.1);
	\draw[fill=black] (3,6) circle(.1);
	\draw[fill=black] (2,4) circle(.1);
	\draw[fill=black] (3, 0) circle(.1);
	\draw[fill=black] (3,-1) circle(.1);
	\draw[fill=black] (3,-2) circle(.1);
	\draw[fill=black] (3,-3) circle(.1);
	\node at (2.45,6) {$3$};
	
	\draw (6,-2) -- (6,6);
	\draw (5,4) -- (6,4);
	
	\draw[fill=black] (6,2) circle(.1);
	\draw[fill=black] (6,3) circle(.1);
	\draw[fill=black] (6,4) circle(.1);
	\draw[fill=black] (6,5) circle(.1);
	\draw[fill=black] (6,6) circle(.1);
	\draw[fill=black] (5,4) circle(.1);
	\draw[fill=black] (6, 1) circle(.1);
	\draw[fill=black] (6,-1) circle(.1);
	\draw[fill=black] (6,-2) circle(.1);
	\draw[fill=black] (6, 0) circle(.1);
	\node at (5.45,6) {$3$};
	
	\draw (9,-1) -- (9,6);
	\draw (8,4) -- (9,4);
	
	\draw[fill=black] (9,3) circle(.1);
	\draw[fill=black] (9,4) circle(.1);
	\draw[fill=black] (9,5) circle(.1);
	\draw[fill=black] (9,6) circle(.1);
	\draw[fill=black] (8,4) circle(.1);
	\draw[fill=black] (9,-1) circle(.1);
	\draw[fill=black] (9, 0) circle(.1);
	\draw[fill=black] (9, 1) circle(.1);
	\draw[fill=black] (9,2) circle(.1);
	\node at (8.45,6) {$3$};
	
	\draw (12,0) -- (12,6);
	\draw (11,4) -- (12,4);
	
	\draw[fill=black] (12,4) circle(.1);
	\draw[fill=black] (12,5) circle(.1);
	\draw[fill=black] (12,6) circle(.1);
	\draw[fill=black] (11,4) circle(.1);
	\draw[fill=black] (12,3) circle(.1);
	\draw[fill=black] (12,2) circle(.1);
	\draw[fill=black] (12,1) circle(.1);
	\draw[fill=black] (12,0) circle(.1);
	\node at (11.45,6) {$3$};
	
	\draw (15,1) -- (15,6);
	\draw (14,4) -- (15,4);
	
	\draw[fill=black] (15,5) circle(.1);
	\draw[fill=black] (15,6) circle(.1);
	\draw[fill=black] (14,4) circle(.1);
	\draw[fill=black] (15,3) circle(.1);
	\draw[fill=black] (15,2) circle(.1);
	\draw[fill=black] (15,1) circle(.1);
	\draw[fill=black] (15,4) circle(.1);
	\node at (14.45,6) {$3$};
	
	\draw (18,2) -- (18,6);
	\draw (17,4) -- (18,4);
	
	\draw[fill=black] (18,5) circle(.1);
	\draw[fill=black] (18,6) circle(.1);
	\draw[fill=black] (17,4) circle(.1);
	\draw[fill=black] (18,3) circle(.1);
	\draw[fill=black] (18,2) circle(.1);
	\draw[fill=black] (18,4) circle(.1);
	\node at (17.45,6) {$3$};
	
	\draw (21,3) -- (21,6);
	\draw (20,4) -- (21,4);
	
	\draw[fill=black] (21,5) circle(.1);
	\draw[fill=black] (21,6) circle(.1);
	\draw[fill=black] (20,4) circle(.1);
	\draw[fill=black] (21,3) circle(.1);
	\draw[fill=black] (21,4) circle(.1);
	\node at (20.45,6) {$3$};
	
	\draw (24,4) -- (24,6);
	\draw (23,4) -- (24,4);
	
	\draw[fill=black] (24,5) circle(.1);
	\draw[fill=black] (24,6) circle(.1);
	\draw[fill=black] (23,4) circle(.1);
	\draw[fill=black] (24,4) circle(.1);
	\node at (23.45,6) {$3$};
	
	\draw (27,5) -- (27,6);
	
	\draw[fill=black] (27,5) circle(.1);
	\draw[fill=black] (27,6) circle(.1);
	\draw[fill=black] (26,4) circle(.1);
	\node at (26.45,6) {$3$};

	\draw (30,5) -- (30,6);

	\draw[fill=black] (30,5) circle(.1);
	\draw[fill=black] (30,6) circle(.1);
	\node at (29.45,6) {$3$};
	\node at (29.45,5) {$4$};
	
	\node at (0,7.2) {$\mathscr{C}_{1}=\Gamma_{12}\qquad$};
	\node at (3,7.2) {$\mathscr{C}_{2}$};
	\node at (6,7.2) {$\mathscr{C}_{3}$};
	\node at (9,7.2) {$\mathscr{C}_{4}$};
	\node at (12,7.2) {$\mathscr{C}_{5}$};
	\node at (15,7.2) {$\mathscr{C}_{6}$};
	\node at (18,7.2) {$\mathscr{C}_{7}$};
	\node at (21,7.2) {$\mathscr{C}_{8}$};
    \node at (24,7.2) {$\mathscr{C}_{9}$};
    \node at (27,7.2) {$\mathscr{C}_{10}$};
	\node at (30,7.2) {$\mathscr{C}_{11}$};
\end{tikzpicture}
\end{center}
\caption{The cinquefoil lattices $\mathscr{C}_n$.}\label{f:Cns}
\end{figure}
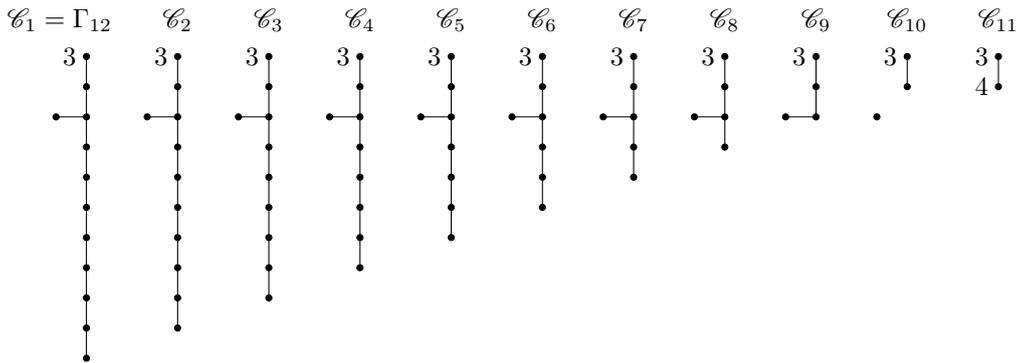
    
    \noindent The lattice $\mathscr{C}_n$ is the positive definite lattice of determinant $n$ and rank $13-n$ formed from the above indicated graph with $n$ vertices; the basis vectors are indexed by the vertices, with squared norm indicated by the weights. Two vectors have pairing $-1$ if their vertices are connected by an edge, and $0$ otherwise. Unmarked vertices have weight $2$. The lattice $\mathscr{C}_{1}$ is the unique positive definite indecomposable unimodular lattice of rank 12, and is called $\Gamma_{12}$. Note also that $\mathscr{C}_{9}=\Lambda(3,2,2,2)$, $\mathscr{C}_{10}=A_1\oplus \Lambda(2,3)$, and $\mathscr{C}_{11}=\Lambda(3,4)$.

Theorem~\ref{thm:2,5} is incomplete in that the authors do not know if the lattices $\Gamma_{12}\oplus \langle n \rangle \oplus \langle 1 \rangle^k$ with $n\in\{2,3\}$ ever occur under the given hypotheses. The obstructions used in this article do not rule these out, but constructions have not yet been found (see Section~\ref{ss:finalrmks}).

The case $n=1$ in Theorem~\ref{thm:trefoil} is a generalization of Fr\o yshov's Theorem, as given in \cite{scaduto-forms}, while the case $n=1$ in Theorem~\ref{thm:2,5} follows from \cite[Theorem 1.1]{scaduto-forms}. The cases $n\geqslant 8$ in Theorem~\ref{thm:trefoil} and $n\geqslant 12$ in Theorem~\ref{thm:2,5}, for which only $\langle n\rangle\oplus \langle 1\rangle^k$ occur, also follow by a computation of $d$-invariants and a theorem of Owens and Strle \cite{owensstrle-delta}; we give a different argument.

The problem of \emph{existence} of~\emph{negative} definite fillings for positive surgeries on knots, without the assumption of having torsion-free homology, was studied by Owens and Strle~\cite{OwensStrle}; among other things, they give explicit necessary and sufficient conditions for a Dehn surgery along a torus knot to bound a negative definite 4-manifold. In contrast, our results concern the \emph{classification} of \emph{positive} definite fillings for positive integral surgeries along certain knots.
As for the results in Owens--Strle, our results are in fact invariant under knot concordance and, more generally, integral homology cobordism of rational homology spheres; for example, the conclusion of Theorem~\ref{thm:trefoil} (ii) holds for any 3-manifold $Y$ that is integer homology cobordant to $S^3_n(T_{2,3})$, e.g. $Y$ is obtained as $n$-surgery along a knot $K$ concordant to $T_{2,3}$.

Note that in Theorem~\ref{thm:trefoil} (ii) we exclude the trivial lattice, i.e. the case when $n=1$ and $k=0$;
indeed, this corresponds to the case when $S^3_1(K)$ bounds an integer homology ball (i.e. it is trivial in the integer homology cobordism group), and in that case none of the trefoil lattices $\mathscr{T}_n$ appear for any $n$.
We also remark that there are knots with slice genus 1 for which $S^3_1(K)$ does bound an integer homology ball: the knot $K = 11n_{79}$ was shown by Akbulut to be $+1$-shake-slice (e.g. there is a sphere representing a generator of the second homology of the trace of $+1$-surgery along $K$), but not slice~\cite{Akbulut}.



We also observe that the assumption that the 4-manifolds have torsion-free homology in Theorem~\ref{thm:trefoil} is essential:
indeed, for instance, both $S^3_{4}(T_{2,3})$ and $S^3_{9}(T_{2,3})$ bound rational homology 4-balls (e.g. see~\cite[Theorem 1.4]{AcetoGolla}), which have trivial intersection form, since $H_2$ is torsion.
By means of Kirby calculus, it is easy to exhibit such rational homology 4-balls with first homology isomophic to $\Z/2\Z$ and $\Z/3\Z$, respectively.

The main content of the article naturally falls under the purview of two headers: {\emph{obstructions}} and {\emph{constructions}}. The obstructions, derived from instanton Floer and Heegaard Floer theory, prove parts (i) of Theorems~\ref{thm:trefoil} and~\ref{thm:2,5}. In fact, the only input from instanton theory in the proof of Theorem~\ref{thm:trefoil} (i) is the above-mentioned result  \cite[Theorem 1.3]{scaduto-forms}, which covers the case $n=1$. We then proceed as follows. Between two consecutive positive integer surgeries there is a 2-handle cobordism; this allows us to argue inductively, realizing a given positive definite lattice bounded by $n$-surgery as the orthogonal complement of a vector with squared norm $n(n-1)$ within a positive definite lattice bounded by $(n-1)$-surgery. This algebraic constraint is not enough: the Heegaard Floer correction terms of Ozsv\'{a}th and Szab\'{o} finish the job; in fact, this input is only needed for the case $n=2$. We use the surgery formula for the Ozsv\'{a}th--Szab\'{o} $d$-invariants due to Ni and Wu \cite[Proposition~1.6]{niwu} and an inequality of Rasmussen \cite[Theorem 2.3]{rasmussen} along the way.
 The proof of Theorem~\ref{thm:2,5} (i) is similar; the only input from instanton theory is for the case $n=1$, from  \cite[Theorem 1.1]{scaduto-forms}, and the Heegaard Floer correction terms are used for the case $n=4$.
 
To prove parts (ii) of Theorems~\ref{thm:trefoil} and~\ref{thm:2,5} we construct smooth 4-manifolds with prescribed boundary and intersection form. While the constructions realizing the lattices in Theorem~\ref{thm:trefoil} (ii) are standard, those for some lattices in Theorem~\ref{thm:2,5} (ii) seem new, and use ideas coming from the topology of plane algebraic curves with cuspidal singularities.

Finally, we remark on the incompleteness of Theorem~\ref{thm:2,5} (ii). The relevant instanton theory is rather undeveloped for rational homology 3-spheres with non-zero $H_1$. The authors were able to use some of the partially developed theory from \cite{froyshov-thesis} to define obstructions in special cases, but could not obstruct the lattices  $\Gamma_{12}\oplus \langle n \rangle \oplus \langle 1 \rangle^k$ with $n\in\{2,3\}$.
It is possible that further development of the theory yields useful obstructions not obtained in this paper.

\subsubsection*{Outline.} In Section~\ref{sec:lattices} we describe basic properties of the lattices $\mathscr{T}_n$ and $\mathscr{C}_n$ that appear in Theorems~\ref{thm:trefoil} and~\ref{thm:2,5}. In Section~\ref{sec:obstructions} we use obstructions to prove parts (i) of Theorems~\ref{thm:trefoil} and~\ref{thm:2,5}, and in Section~\ref{sec:constructions} we provide our constructions to prove parts (ii) of Theorems~\ref{thm:trefoil} and~\ref{thm:2,5}.

\subsubsection*{Acknowledgments.}
We like to thank Carlos Rito for his assistance with computations with rational cuspidal curves (which were unsuccessful and did not make it to the paper). We also thank the organizers of the July 2018 conference ``Gauge Theory and Applications'' at the University of Regensburg for their hospitality and support.

Finally, we thank the referee for their useful comments and suggestions.

\section{Distinguished lattices}\label{sec:lattices}

Let $L$ be an integral lattice. Write $x\cdot y\in \Z$ for the pairing on $L$ applied to $x,y\in L$, and $x^2=x\cdot x$. The dual lattice $L^\ast$ is defined to be the subset of $L\otimes \Q$ consisting of vectors $x$ such that $x\cdot y\in\Z$ for all $y\in L$. The {\emph{determinant}} or {\emph{discriminant}} of $L$ is the number $[L^\ast:L]=|L^\ast/L|$, and agrees with the absolute value of the determinant of a matrix representation for the pairing defined on $L$. Henceforth, a lattice should be assumed integral, unless it arises as the dual of an integral lattice.

One of the lattices that we will encounter is the odd, unimodular, indefinite lattice ${\rm I}_{a,b}$; this is diagonalizable, and more precisely it admits an orthogonal basis $h_1,\dots,h_a, e_1,\dots,e_b$ such that $h_i^2 = 1$ and $e_j^2 = -1$.
If $a=1$, we will simply write $h$ instead of $h_1$.

Let $L$ be a positive definite lattice. We say $L$ is {\emph{reduced}} if it has no vectors of norm 1. We may always write $L=L'\oplus \langle 1\rangle^k$ where $k\geqslant 0$ and $L'$ is reduced; we call $L'$ the reduced part of $L$. A {\emph{root}} in $L'$ is a vector of squared norm 2. The {\emph{root lattice}} $\mathsf{R}(L)\subset L'$ is the sublattice generated by roots. A classical result of Witt says any root lattice is a direct sum of the lattices $A_n$ ($n\geqslant 1)$, $D_n$ ($n\geqslant 4$), $E_6$, $E_7$, $E_8$. For the basic properties of these lattices, see \cite[Chapter~4]{conwaysloane}.

A coordinate system for $E_8$ is obtained by taking the lattice in $\R^8$ generated by $(x_1,\ldots,x_8)\in\Z^8$ with $\sum x_i \in 2\Z$ and $(\nicefrac{1}{2},\ldots,\nicefrac{1}{2})$. A corresponding root diagram for $E_8$ is:
\begin{center}
\begin{tikzpicture}
	\draw (-1,0) -- (5,0);
	\draw (3,0) -- (3,1);
	
	\draw[fill=black] (1,0) circle(.06);
	\draw[fill=black] (2,0) circle(.06);
	\draw[fill=black] (3,1) circle(.06);
	\draw[fill=black] (3,0) circle(.06);
	\draw[fill=black] (4,0) circle(.06);
	\draw[fill=black] (0,0) circle(.06);
	\draw[fill=black] (-1,0) circle(.06);
	\draw[fill=black] (5,0) circle(.06);
	
	\node[font=\tiny] at (-1,-.35) {$(-\nicefrac{1}{2},\nicefrac{1}{2}^6,-\nicefrac{1}{2})$};
	\node[font=\tiny] at (0,.35) {$(0^6,-1,1)$};
	\node[font=\tiny] at (1,-.35) {$(0^5,-1,1,0)$};
	\node[font=\tiny] at (2,.35) {$(0^4,-1,1,0^2)$};
	\node[font=\tiny] at (3,-.35) {$(0^3,-1,1,0^3)$};
	\node[font=\tiny] at (4,.35) {$(0^2,-1,1,0^4)$};
	\node[font=\tiny] at (3,1.35) {$(\nicefrac{1}{2}^4,-\nicefrac{1}{2}^4)$};
	\node[font=\tiny] at (5,-.35) {$(0,-1,1,0^5)$};
\end{tikzpicture}
\end{center}
Superscripts are to be interpreted as repeated entries. In the above graph, all vertices have weight 2. The lattice $E_8$ is unimodular, i.e. $\det(E_8)=1$.

The lattice $E_7$ is obtained by taking the complement of $(\nicefrac{1}{2}^8)\in E_8$. A root diagram for $E_7$ is obtained by removing the far left node of the above graph, and $\det(E_7)=2$. The lattice $E_6$ is obtained as the complement of $(\nicefrac{3}{2},-\nicefrac{1}{2}^6,\nicefrac{3}{2})\in E_7$, a vector of squared norm $6$. A root diagram for $E_6$ is obtained by removing the two far left nodes in the above graph for $E_8$, and $\det(E_6)=3$. 

The lattice $D_n$ ($n\geqslant 4$) is defined to be the set of $(x_1,\ldots,x_{n})\in\Z^n$ such that $\sum x_i\in 2\Z$. Writing $e_1,e_2,\ldots$ for the standard orthonormal basis of Euclidean space, a root diagram for $D_n$ is:
\begin{center}
\begin{tikzpicture}
	\draw (-1,0) -- (2,0);
	\draw (3,0) -- (4,0);
	\draw (4,0) -- (5,.5);
	\draw (4,0) -- (5,-.5);
	
	\draw[fill=black] (1,0) circle(.06);
	\draw[fill=black] (5,-.5) circle(.06);
	\draw[fill=black] (4,0) circle(.06);
	\draw[fill=black] (0,0) circle(.06);
	\draw[fill=black] (-1,0) circle(.06);
	\draw[fill=black] (5,.5) circle(.06);
	
	\node at (2.5,0) {$\ldots$};
	\node[font=\tiny] at (-1,-.35) {$e_1-e_2$};
	\node[font=\tiny] at (0,.35) {$e_2-e_3$};
	\node[font=\tiny] at (1,-.35) {$e_3-e_4$};
	\node[font=\tiny] at (3.6,-.35) {$e_{n-2}-e_{n-1}$};
	\node[font=\tiny] at (5.8,.5) {$e_{n-1}-e_n$};
	\node[font=\tiny] at (5.8,-.5) {$e_{n-1}+e_n$};
\end{tikzpicture}
\end{center}
We have $\det(D_n)=4$. In fact, $D_n^\ast/D_n$ is isomorphic to $\Z/4\Z$ if $n$ is odd and $\Z/2\Z\oplus \Z/2\Z$ if $n$ is even. Next, $D_5$ is isomorphic to the complement of $(-\nicefrac{3}{2},\nicefrac{1}{2}^5,-\nicefrac{5}{2},\nicefrac{3}{2})\in E_6$, a vector of square $12$; a basis of roots is obtained from the above graph for $E_8$ by removing the three far left nodes.

The lattice $A_n$ is the set of $(x_1,\ldots,x_{n+1})\in\Z^{n+1}$ with $\sum x_i=0$. A root diagram for $A_n$ is given by $n$ vertices in a straight line graph, represented by, say, the roots $e_i-e_{i+1}$ for $1\leqslant i\leqslant n$. We have $\det(A_n)=n+1$. Furthermore, $A_4$ is isomorphic to the orthogonal complement of $(2,2,2,2,2)\in D_5$, a vector of squared norm $20$. Next, $A_1\oplus A_2$ is isomorphic to the complement of the vector $(3,3,-2,-2,-2)\in A_4$ of squared norm $30$. Finally, $\Lambda(2,4)$ is isomorphic to the complement of the vector $((3,-3),(2,2,-4))\in A_1\oplus A_2$ of square $42$.

The above descriptions are unique in the following sense. Recall that we write $\mathscr{T}_1,\ldots,\mathscr{T}_7$ for the above distinguished lattices $E_8,E_7,E_6,D_5,A_4,A_1\oplus A_2,\Lambda(2,4)$. Note $\det(\mathscr{T}_n)=n$. It is convenient for the following statement to set $\mathscr{T}_8=\langle 8\rangle$.

\begin{lemma}\label{lemma:distinguished}
    Let $1\leqslant n\leqslant 7$. If $L=v^\perp \subset \mathscr{T}_n$, where $v^2=n(n+1)$, $\det(L)=n+1$, then $L\cong \mathscr{T}_{n+1}$.
\end{lemma}

\begin{proof}
It is well-known that the automorphism group of $E_8$ acts transitively on roots, so the complement of any root is isomorphic to $E_7$. This covers the case $n=1$. The remaining cases may be seen directly. For example, let $n=4$. A vector in $D_5$ of squared norm $20$ is equivalent to one of $(2^5)$, $(4,1^4)$, $(3^2,1^2,0)$, or $(4,2,0^3)$. The first case yields $A_4$. For the cases $(4,1^4)$ and $(3^2,1^2,0)$, note that the complements taken within the larger lattice $\Z^5$ have determinant $20$. Since $v^\perp\subset D_5$ is a full-rank sublattice of this larger complement, the determinant is divisible by $20$, hence cannot be $5$. For the case of $(4,2,0^3)=2(2,1,0^3)$, the complement in $\Z^5$ has determinant $5$. As $v^\perp\subset D_5$ is a proper full-rank sublattice therein, its determinant must be strictly larger than $5$.

The other cases are also straightforward. Alternatively, the lemma follows from Lemma \ref{lemma:evenlattices} below, which follows from the classification of positive definite integral lattices of low determinant and rank, as given in \cite[Table 0]{cs-low}, and reproduced in Table~\ref{fig:tableroots}.
\end{proof}

\begin{lemma}\label{lemma:evenlattices}
    If $L$ is an even positive definite lattice with $\det(L)=9-\text{\emph{rk}}(L)$, then $L\cong \mathscr{T}_{\det(L)}$.
\end{lemma}

Now we consider the lattices $\mathscr{C}_n$. The unimodular lattice in this family, $\mathscr{C}_1=\Gamma_{12}$, may be defined as the lattice in $\R^{12}$ generated by $D_{12}$ along with the vector $(1/2,\ldots,1/2)$. The plumbing description is obtained from this by choosing the vertices to be the following vectors:
\begin{center}
\begin{tikzpicture}
	\draw (-1,0) -- (9,0);
	\draw (1,0) -- (1,1);
	
	\draw[fill=black] (1,0) circle(.06);
	\draw[fill=black] (2,0) circle(.06);
	\draw[fill=black] (1,1) circle(.06);
	\draw[fill=black] (3,0) circle(.06);
	\draw[fill=black] (4,0) circle(.06);
	\draw[fill=black] (0,0) circle(.06);
	\draw[fill=black] (-1,0) circle(.06);
	\draw[fill=black] (5,0) circle(.06);
	\draw[fill=black] (6,0) circle(.06);
	\draw[fill=black] (7,0) circle(.06);
	\draw[fill=black] (8,0) circle(.06);
	\draw[fill=black] (9,0) circle(.06);
	
	\node[font=\tiny] at (-1,-.35) {$(\nicefrac{1}{2}^{12})$};
	\node[font=\tiny] at (0,.35) {$(-1,-1,0^{10})$};
	\node[font=\tiny] at (1,-.35) {$(0,1,-1,0^{9})$};
	\node[font=\tiny] at (2,.35) {$(0^{2},1,-1,0^{8})$};
	\node[font=\tiny] at (3,-.35) {$(0^3,-1,1,0^7)$};
	\node[font=\tiny] at (4,.35) {$(0^4,-1,1,0^6)$};
	\node[font=\tiny] at (1,1.35) {$(1,-1,0^{10})$};
	\node[font=\tiny] at (5,-.35) {$(0^5,1,-1,0^5)$};
	\node[font=\tiny] at (6,.35) {$(0^6,1,-1,0^4)$};
	\node[font=\tiny] at (7,-.35) {$(0^7,1,-1,0^3)$};
	\node[font=\tiny] at (8,.35) {$(0^8,1,-1,0^2)$};
	\node[font=\tiny] at (9,-.35) {$(0^9,1,-1,0)$};
\end{tikzpicture}
\end{center}
The only node not of weight 2 is the far left node of weight 3. It is convenient for the following statement to set $\mathscr{C}_{12}=\langle 12\rangle$.

\begin{lemma}\label{lemma:clattices}
    Let $1\leqslant n \leqslant 11$. If $L=v^\perp \subset \mathscr{C}_n$, where $v^2=n(n+1)$, $\det(L)=n+1$, and $n\neq 3$, then $L\cong \mathscr{C}_{n+1}$. If $n=3$ then $L\cong \mathscr{C}_{4}$ or $L\cong D_9$. There is no $v\in D_9$ with $v^2=20$ and $\det(v^\perp)=5$.
\end{lemma}

\begin{proof}
We refer again to Table \ref{fig:tableroots}, now for the classification of reduced positive definite integral lattices with determinant $n$ and rank $13-n$. Note that the table only lists {\emph{indecomposable}} lattices, and that all lattices therein are listed by their root lattices. When the root lattice is of positive codimension $i$, the symbol $O_i$ is included in the notation.
Observe that, for $2\leqslant n \leqslant 8$, the root lattice of $\mathscr{C}_n$ is $D_{12-n}$.

We start with $n=1$. The lattice listed in Table \ref{fig:tableroots} as $D_{12}$ is none other than $\mathscr{C}_1=\Gamma_{12}$, and there is only one possibility for $L=v^\perp$; indeed, according to the table, there is a unique positive definite lattice of rank 11 and determinant 2, $D_{10}O_1$, which must be isomorphic to $\mathscr{C}_2$. Similarly, for $n=2$, the only possibility for $L$ is $\mathscr{C}_3=D_9O_1$. For $n=3$, however, there are several possibilities for $L$, namely $\mathscr{C}_4=D_8O_1$, $D_9$, $E_7A_1O_1$, and $E_8\oplus A_1\oplus A_1$. The latter two lattices cannot occur, as $E_7$ and $E_8$ do not embed into $D_9$, the root lattice of $\mathscr{C}_{3}$. To see that $D_9$ can occur, view $\mathscr{C}_{3}$ as the lattice generated by $(\nicefrac{1}{2}^{12})$, $(-1^2,0^{10})$, $(1,-1,0^{10})$, $\ldots$, $(0^{7},1,-1,0^3)$.  Then $v=(0^9,2^3)\in\mathscr{C}_{3}$ has $v^2=12$ and $v^\perp\cong D_9$; indeed, $v^\perp$ is spanned by $(-1^2,0^{10})$, $(1,-1,0^{10})$, $\ldots$, $(0^{7},1,-1,0^3)$.

Next, if $n=4$, there are a priori two possibilities for $L$ according to Table \ref{fig:tableroots}, $\mathscr{C}_5=D_7O_1$ and $E_7O_1$. However, the latter cannot occur, as $E_7$ does not embed into $D_8$, the root lattice of $\mathscr{C}_4=D_8O_1$. Similarly, if $n\in\{5,6,7,8\}$, there are in each case two possibilities: for $n=5$, $\mathscr{C}_6=D_6O_1$ or $E_6\oplus A_1$; for $n=6$, $\mathscr{C}_6=D_5O_1$ or $A_6$; for $n=7$, $\mathscr{C}_8=D_4O_1$ or $D_4\oplus A_1$; and for $n=8$, $\mathscr{C}_9=A_3O_1$ and $A_2\oplus A_2$. In each case, the lattice listed after $\mathscr{C}_{n+1}$ cannot occur, because it has a root lattice that does not embed into $D_{12-n}$, the root lattice of $\mathscr{C}_n$.

If $n=9$, then $L$ is either $\mathscr{C}_{10}=A_1\oplus\Lambda(2,3)$ or $A_2O_1=\Lambda(2,2,4)$. We claim that the latter lattice cannot occur. Note that $\mathscr{C}_9=\Lambda(2,2,2,3)$ embeds into $\mathbb{Z}^{6}$ as the sublattice generated by $(1^2,-1,0^3)$, $(0^2,1,-1,0^2)$, $(0^3,1,-1,0)$, $(0^4,1,-1)$. From this description, it is easily seen that the vectors of squared norm 4 in $\mathscr{C}_9$ are in the root lattice $A_3\subset \mathscr{C}_9$. Consequently, if $L$ were isomorphic to $\Lambda(2,2,4)$, then $\Lambda(2,2,4)$ would embed into $A_3$. However, the determinant of $\Lambda(2,2,4)$ is 10, which is not divisible by $\det(A_3)=4$, a contradiction, verifying the claim.

If $n=10$, then $L$ is either $\mathscr{C}_{11}=O_2$ or $A_1O_1=\Lambda(2,6)$. If the latter were to occur, it would embed into the $\Lambda(2,3)$-summand of $\mathscr{C}_{10}$. However, the determinant of $A_1O_1$ is 11, and is not divisible by $\det(\Lambda(2,3))=5$. Thus $A_1O_1$ cannot occur. If $n=11$, the only possibility for $L$ is $\langle 12\rangle$.

Finally, if $v\in D_9$ with $v^2=20$ and $\det(v^\perp)=5$, then $v^\perp$ must be one of $D_7O_1$ or $E_7O_1$. However, the lattice $\mathscr{C}_4=D_7O_1$ cannot occur because it is odd, while $v^\perp$ must be even; and the lattice $E_7O_1$ cannot occur because $E_7$ does not embed into $D_9$. Thus there is no such vector $v$.
\end{proof}


\section{Obstructions}\label{sec:obstructions}

In this section we prove Theorem~\ref{thm:trefoil} (i) and Theorem~\ref{thm:2,5} (i). The obstructions used are introduced in Sections~\ref{subsec:surgerycob}--\ref{subsec:instantons}, and the proofs are completed in Section~\ref{subsec:obsproofs}. All homology groups are taken with integer coefficients, unless specified otherwise.

Let $X$ and $Y$ be as in Definition~\ref{def:filling}, so that the intersection form $L_X=H_2(X)$ fills $Y$. As $H_2(Y)=0$, the long exact sequence for the pair $(X,Y)$ yields the exact sequence
\begin{equation}
0\longrightarrow H_2(X) \longrightarrow H_2(X,Y) \longrightarrow H_1(Y) \longrightarrow H_1(X)\label{eq:les}
\end{equation}

Since $H_\ast(X)$ has no torsion, the right-most map in~\eqref{eq:les} goes from the torsion group $H_1(Y)$ to the free abelian group $H_1(X)$, so it must be zero.
The map $H_2(X)\to H_2(X,Y)$ is an isomorphism over the rationals, and so the pairing on $H_2(X)$ may be extended as a $\Q$-valued pairing on $H_2(X,Y)$.
An application of Poincar\'{e}--Lefschetz duality shows that $H_2(X,Y)$ may be identified with the dual lattice $L_X^\ast$.
In this way, \eqref{eq:les} induces an isomorphism between $H_1(Y)$ and the discriminant group $L_X^\ast /L_X$, and in particular $\det L_X = |H_1(Y)|$.

\subsection{Algebraic constraints from surgery cobordisms}\label{subsec:surgerycob}

We now discuss constraints imposed by surgery 2-handle cobordisms. Let $K$ be a knot in the 3-sphere, and let $Y_n = S^3_n(K)$ be the result of $n$-surgery on $K$ where $n\in\Z$, and $K_n\subset Y_n$ be the dual knot.
For $n\geqslant 2$ there is a positive definite surgery cobordism $W_n:Y_{n}\to Y_{n-1}$ obtained by attaching a 2-handle to a meridian of the surgered neighborhood of $K_n\subset Y_{n}\times\{1\}\subset Y_{n}\times [0,1]$ with framing $+1$, as in Figure~\ref{fig:trefoilcob}.
From the long exact sequence of the pair $(W_n,Y_n)$, using the fact that $H_*(W_n,Y_n)$ is free abelian of rank 1, supported in degree 2, by Poincar\'e--Lefschetz duality we see that the group $L_n:=H_2(W_n)$, too, is free abelian of rank 1.
Observe that $L_{n}^\ast/L_{n}$ is isomorphic to $H_1(Y_n)\oplus H_1(Y_{n-1})$, which is cyclic of order $n(n-1)$. As $L_{n}$ is rank 1 and positive definite, it must be isomorphic to $\langle n(n-1)\rangle$.

\begin{figure}[t]
\labellist
\pinlabel $\langle n\rangle$ at 40 250
\pinlabel $+1$ at 380 50
\pinlabel $n-1$ at 625 250 
\endlabellist
\centering
\includegraphics[scale=.30]{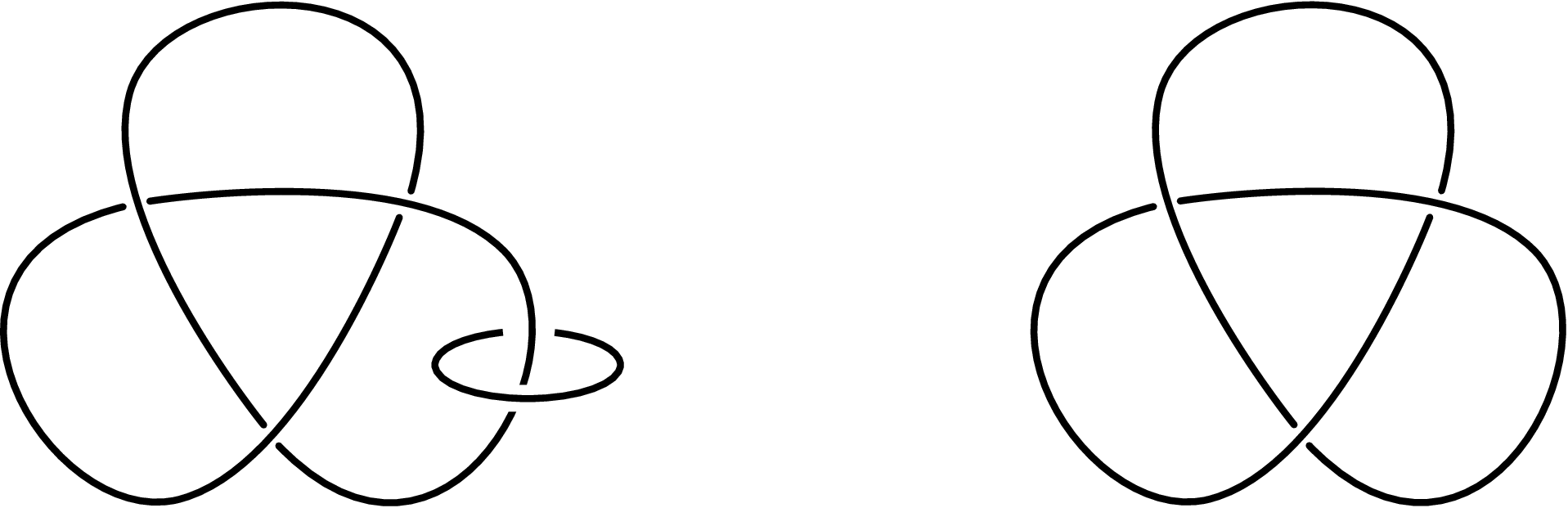}
\caption{\small The surgery cobordism $W_n$ is obtained by attaching a $+1$-framed meridian to $n$-surgery on $K$ (left), and the outgoing boundary of $W_n$ is homeomorphic to $(n-1)$-surgery (right). 
Here we use angular braces for relative handlebodies~\cite[Section 5.5]{GS}.}\label{fig:trefoilcob}
\end{figure}

Now suppose a positive definite smooth 4-manifold $X$ with no torsion in its homology has boundary $Y_n$. Then the Mayer--Vietoris sequence for the composite $Z=X\cup W_n$ yields
\[
0\longrightarrow H_2(X) \oplus H_2(W_n) \longrightarrow H_2(Z) \longrightarrow H_1(Y_n) \longrightarrow 0.
\]
Via Poincar\'{e}--Lefschetz duality we obtain an embedding of lattices $L_X\oplus L_n\hookrightarrow L_Z$.

\begin{lemma}\label{lemma:cobordism}
Let $n\geqslant 2$. Suppose $S_{n}^3(K)$ bounds a positive definite lattice $L$. Then $\det L=n$ and $L=v^\perp \subset M$ where $M$ is positive definite, fills $S_{n-1}^3(K)$, and $v\in M$ has $v^2=n(n-1)$.
\end{lemma}

\begin{proof}
As seen at the beginning of the section, $\det L = n$.

Suppose $L = L_X$; we will show we can choose $M = L_Z$ and $v$ to be the image of the generator of $L_n$ under the map induced by the inclusion $W_n\hookrightarrow Z$, as constructed above.

We begin by showing that $Z$ has torsion-free homology, which shows that $L_Z$ fills $Y_{n-1}$, and in particular that $\det L_Z = n-1$.
Indeed, we can look at the long exact sequence of the pair $(Z,X)$; note that $H_*(Z,X) = H_*(W_n,Y_n)\cong \Z$, supported in degree 2, by excision.
Therefore, the inclusion $X\hookrightarrow Z$ induces an isomorphism in degrees different from 1 and 2;
in those degrees, we have:
\[
0 \longrightarrow H_2(X) \longrightarrow H_2(Z) \longrightarrow H_2(Z,X) \longrightarrow H_1(X) \longrightarrow H_1(Z) \longrightarrow 0.
\]
However, since $H_2(Z,X) \cong \Z$, and $H_2(X)$ has co-rank 1 in $H_2(Z)$, we immediately see that the map $H_2(Z) \to H_2(Z,X)$ is surjective, hence $H_1(X) \cong H_1(Z)$;
at the same time, $H_2(Z)$ sits in an exact sequence between two free groups, hence it is itself free.

We also have that $L_X \subset v^\perp$ as a full rank sub-lattice, since we have an embedding $L_X\oplus L_n \hookrightarrow L_Z$ and $L_X$ has co-rank 1 in $L_Z$.
Finally, since $L_Z/v^\perp = (L_Z/L_X)/(v^\perp/L_X)$ and $L_Z/L_X \cong \Z$, $v^\perp/L_X = 0$, i.e. $L_X = v^\perp$.
\end{proof}

\begin{lemma}\label{lemma:littlelattice}
    Suppose $M$ is a lattice, $v\in M$ has $v^2\neq 0$, and let $d=\text{{\emph{gcd}}}\left\{v\cdot w : w\in M\right\}$. Then 
    \[
        \det(v^\perp) \; = \; \frac{v^2}{d^2}\cdot \det(M).
    \]  
\end{lemma}

\begin{proof}
    Write $\ell = v^2=v\cdot v$. Define $N\subset M$ as $N = \{m + kv : m\in v^\perp, k\in \mathbb{Z}\}$. As $v$ is orthogonal to $v^\perp$, $N$ is isomorphic to $v^\perp\oplus\langle \ell\rangle$; it follows that $\det(N) = \ell \det(v^\perp)$. Define $f: M \to \mathbb{Z}$ by $f(w) = w\cdot v$; then $\ker f = v^\perp$, and $f(M) = d\mathbb{Z}$, while $f(N) = \ell\mathbb{Z}$.
It follows by the third isomorphism theorem that $M/N = \mathbb{Z}/(\ell/d)\mathbb{Z}$, and in particular $[M:N] = \ell/d$. We now recall that
\begin{align*}
\det(N) &= [N^* : N] = [N^* : M^*] \cdot  [M^* : M] \cdot [M : N]\\
& = [N^* : M^*]\cdot \det(M) \cdot [M : N] \\
& = [M : N]^2 \cdot \det(M),
\end{align*}
from which we obtain $\det(v^\perp) = \det(N)/\ell = (\ell/d^2)\cdot\det(M)$, as claimed.
\end{proof}

\subsection{Heegaard Floer correction terms}\label{subsec:dinvs}

We now recall the relevant background for Heegaard Floer $d$-invariants defined by Ozsv\'{a}th and Szab\'{o}~\cite{os}, building on analogous work of Fr{\o}yshov in the monopole setting~\cite{Froyshov-correction}. Each such {\emph{correction term}} $d(Y,\mathfrak{t})\in\Q$ is an invariant of $(Y,\mathfrak{t})$ where $Y$ is a rational homology 3-sphere and $\mathfrak{t}$ is a spin$^\text{c}$ structure for $Y$. Oszv\'{a}th and Sz\'{a}bo prove the following inequality \cite[Theorem~9.6]{os}: if $X$ is a smooth, compact, oriented, positive definite 4-manifold bounded by a rational homology 3-sphere $Y$, and $\mathfrak{s}$ is a spin$^\text{c}$ structure on $X$, then
\begin{equation}
      -4d(Y,\mathfrak{t}) \; \geqslant\;  b_2(X)-c_1(\mathfrak{s})^2\label{eq:dinvineq},
\end{equation}
where $\mathfrak{t}$ is the restriction of $\mathfrak{s}$ to $Y$. A {\emph{characteristic covector}} for a lattice $L$ is an element $\xi\in L^\ast$ such that $\xi\cdot x \equiv x^2 $ (mod 2) for all $x\in L$. Let $\text{Char}(L)$ denote the set of characteristic covectors for $L$. If the 4-manifold $X$ has no torsion in its homology, then $c_1(\mathfrak{s})$ ranges over the set of characteristic covectors for $L_X$ as $\mathfrak{s}$ varies.

\begin{defn} For a rational homology $3$-sphere $Y$, we set $\delta(Y)=-4\min_{\mathfrak{t}\in\text{\emph{Spin}}^c(Y)} d(Y,\mathfrak{t})$.
For a positive definite lattice $L$, we set $\delta(L)=\text{\emph{rk}}(L)-\min_{\xi\in \text{\emph{Char}}(L)}|\xi^2|$.
\end{defn}

For a positive definite lattice $L$ that fills a rational homology 3-sphere $Y$ inequality~\eqref{eq:dinvineq} implies
\begin{equation}
	\delta(Y) \; \geqslant \; \delta(L).\label{eq:delta}
\end{equation}

The lattice invariant $\delta(L)$ satisfies the following, as is easily verified: $\delta(L_1\oplus L_2)=\delta(L_1)+\delta(L_2)$; $\delta(L)=\text{rk}(L)$ if $L$ is even; $\delta(\langle n\rangle)=(n-1)/n$ if $n$ is odd. Furthermore, $\delta(\Gamma_{12})=8$. Indeed, using the description of $\Gamma_{12}$ from Section~\ref{sec:lattices}, a minimal characteristic covector for $\Gamma_{12}$ is given by $(2,0^{11})$.

A family of knot invariants $V_i(K)\in \Z_{\geqslant 0}$, indexed by $i\in \Z_{\geqslant 0}$, was introduced by Rasmussen~\cite{rasmussen2003floer} (who used the notation $h_i(K)$ instead of $V_i(m(K))$, where $m(K)$ is the mirror of $K$), and later studied by Rasmussen~\cite{rasmussen}, and Ni and Wu~\cite{niwu}. The invariants satisfy $V_{i}(K)-1\leqslant V_{i+1}(K)\leqslant V_i(K)$ for every $i\geqslant 0$, and vanish for large enough $i$. Let $U$ denote the unknot. Then Ni and Wu prove the following~\cite[Proposition~1.6, Remark~2.10]{niwu}:
\begin{equation}
    d\left(S_{p/q}^3(K),i\right) \;  = \; -2\max\left\{V_{\lfloor i/q \rfloor}(K),V_{\lceil (p-i)/q \rceil}(K) \right\} + d\left(S_{p/q}^3(U),i\right)\label{eq:dinvsurgery}
\end{equation}

Here $p/q$ is a positive rational number and $i$ is a non-negative integer less than $p$, corresponding to a spin$^\text{c}$ structure. On the other hand, Rasmussen shows in \cite[Theorem~2.3]{rasmussen} that
\[
    V_i(K) \;\leqslant \; \left\lceil \frac{g_4(K)-|i|}{2}\right\rceil \quad \text{ if }|i| < g_4(K),
\]
while  $V_i(K)= 0$ if $|i|\geqslant g_4(K)$. Here $g_4(K)$ is the slice genus of $K$. In particular, if $g_4(K)\leqslant 2$, there are three possibilities: either $V_i(K)=0$ for all $i\geqslant 0$; $V_0(K)=1$ and $V_i(K)=0$ for $i\geqslant 1$; or $V_0(K)=V_1(K)=1$ and $V_i(K)=0$ for $i\geqslant 2$. These three cases are realized, respectively, by the unknot $U$, the right-handed trefoil $T_{2,3}$, and the $(2,5)$-torus knot $T_{2,5}$. Together with \eqref{eq:dinvsurgery}, this implies the following.

\begin{lemma}\label{lemma:dinvgenus2}
Let $K$ be a knot of slice genus at most $2$. Then $d(S_{p/q}^3(K),i)=d(S_{p/q}^3(K'),i)$ for all $0\leqslant i < p$ and $p/q>0$, where $K'$ is either the unknot $U$, $T_{2,3}$ or $T_{2,5}$.
\end{lemma}

By \cite[Proposition~4.8]{os} we have $d(S_{n}^3(U),i)=((2i-n)^2-n)/4n$ for $0\leqslant i <n$, and along with~\eqref{eq:dinvsurgery} this computes $d(S_n^3(K),i)$ for $K\in\{T_{2,3},T_{2,5}\}$. Relevant to our arguments below are:
\begin{align}
	\delta(S_2^3(T_{2,3})) \; = \; \;\;\max_{i\in\{0,1\}} -4d\left(S_2^3(T_{2,3}),i\right) \; & = \; 7, \label{eq:dinv23}\\
	\delta(S^3_4(T_{2,5})) \; = \; \max_{i\in\{0,1,2,3\}} -4d\left(S_4^3(T_{2,5}),i\right) \; & = \; 8.\label{eq:dinv25}
\end{align}
Further, we also have $\delta(S_4^3(K))<8$ when $K\in\{U,T_{2,3}\}$.

\subsection{Input from instanton theory} \label{subsec:instantons}
Finally, the most important obstruction we use is the following result from Yang--Mills instanton theory. The $g_4(K)=0$ case follows from Donaldson's Theorem~\cite{d-connections}, and the $g_4(K)=1$ case is a generalization of~\cite[Theorem~4.1]{froyshov-thesis}.

\begin{theorem}[\hspace{-0.0001cm}{\cite[Theorems~1.1~and~1.3]{scaduto-forms}}]\label{thm:instanton}
	Suppose a smooth, compact, oriented, and positive definite $4$-manifold $X$ with no $2$-torsion in its homology has boundary $S_n^3(K)$ for some knot $K$ with slice genus $g_4(K)\leqslant 2$. Then $L_X$ is isomorphic, for some $k\geqslant 0$, to one of the following:
\begin{align*}
	g_4(K)=0: \quad\;\; & \langle 1 \rangle^k;\\
	g_4(K)=1:\quad\;\;  & \langle 1\rangle^k, \qquad E_8\oplus\langle 1\rangle^k;\\
	g_4(K)=2:\quad \;\; & \langle 1 \rangle^k, \qquad E_8\oplus \langle 1 \rangle^k, \qquad \Gamma_{12}\oplus \langle 1 \rangle^k.\\
\end{align*}
\end{theorem}

\subsection{Proofs of parts (i) of Theorems~\ref{thm:trefoil} and~\ref{thm:2,5}}\label{subsec:obsproofs}

\begin{proof}[Proof of Theorem~\ref{thm:trefoil} (i)]
The case $n=1$ is implied by Theorem~\ref{thm:instanton}. Let $n\geqslant 2$. Suppose a positive definite lattice fills $Y_n=S_n^3(K)$. By Lemma~\ref{lemma:cobordism}, this lattice is isomorphic to a determinant-$n$ lattice $v^\perp \subset M$ for some $v\in M$ with $v^2=n(n-1)$, where $M$ is positive definite, fills $S_{n-1}^3(K)$, and has determinant $n-1$.
Write $M=M'\oplus\langle 1 \rangle^k$ where $M'$ is reduced, and $v=(x,y)$ where $x\in M'$ and $y=\sum y_i e_i\in\langle 1 \rangle^k$. Write $d=\text{gcd}\{v\cdot w:w\in M\}$. Then by Lemma~\ref{lemma:littlelattice} we have
\begin{equation}
	n \; = \; \det(v^\perp) \; = \; \frac{v^2}{d^2}\cdot \det(M) \; = \; \frac{n(n-1)^2}{d^2} \implies d \; = \; n-1.\label{eq:gcd}
\end{equation}

On the other hand, as $e_i\cdot y=y_i$, $d$ divides $\gcd\{y_i\}$, and so there are $b_i\in \Z$ such that $y_i=(n-1)b_i$. Consequently, $y=(n-1)b$ where $b=\sum b_ie_i$. Observe that $n(n-1) - (n-1)^2b^2=  v^2-y^2  = x^2  \geqslant 0$. In particular, if $n\geqslant 3$, after an automorphism of $M$, we are in one of the following two cases:
\begin{equation}
		y=0 \;\;\quad \text{ or } \;\;\quad y=(n-1)e_1 \qquad (n\geqslant 3).\label{eq:twocases}
\end{equation}
Note that in the first case $x^2=v^2=n(n-1)$, while in the latter case $x^2=v^2-y^2=n-1$.

Suppose $M=\langle 1\rangle^{k}$, or equivalently, $n=2$ and $M'=0$. Then after an automorphism of $M$ we may write $v=y=e_1+e_2$, and consequently $ v^\perp\cong \langle 2 \rangle\oplus \langle 1\rangle^{k-2}$. Next suppose $n\geqslant 3$ and $M'=\langle n-1\rangle$. From~\eqref{eq:twocases}, either $y=0$ or $y=(n-1)e_1$. Suppose $y=0$. Then $x$ is a multiple of a generator for $\langle n-1 \rangle$ and $v^\perp=\langle 1 \rangle^{k-1}$, contradicting $\det(v^\perp)=n$. So we must have $y=(n-1)e_1$. Then $x^2=n-1$ and so $x$ is a generator for $\langle n-1\rangle$. The orthogonal complement $v^\perp$ is generated by $x-e_1$ and $e_j$ ($j\neq 1$), isomorphic to $\langle n\rangle\oplus \langle 1 \rangle^{k-1}$. In summary, if $M'=0$ then $v^\perp\cong \langle 2 \rangle\oplus \langle 1\rangle^{k-2}$, and if $M'=\langle n-1\rangle$ then $v^\perp \cong \langle n\rangle\oplus \langle 1\rangle^{k-2}$.

Consulting Lemma~\ref{lemma:dinvgenus2}, suppose the $d$-invariants of $S_n^3(K)$ agree with those for $S_n^3(U)$. We claim the only lattices that occur are $\langle n\rangle\oplus \langle 1 \rangle^k$. Let us look at the case $n=1$ first; as we have assumed that $K$ has slice genus at most $1$, we know that $M'$ is either zero or $E_8$. However, $\delta(E_8)=\text{rk}(E_8)=8$ and $\delta(S_1^3(U))=-4d(S^3)=0$, contradicting~\eqref{eq:delta}. Thus $M=\langle 1 \rangle^k$. The claim then follows by induction. At each step, in which $n\geqslant 2$, $M'$ is isomorphic to $\langle n-1\rangle$, and from the argument in the previous paragraph, we must have $v^\perp \cong \langle n\rangle\oplus \langle 1\rangle^{k-2}$.

Next, suppose the $d$-invariants of $S_n^3(K)$ agree with those of $S_n^3(T_{2,3})$. The proof is again by induction. At each step, $M'$ is either $\langle n-1 \rangle$ or $\mathscr{T}_{n-1}$. We have dealt with the possibilities when $M'=\langle n-1 \rangle$, so we may always assume $M'=\mathscr{T}_{n-1}$. Let $n=2$. Then $M'=E_8$. Then either $v=x\in E_8$ is a root, or $v=y\in \langle 1 \rangle^k$. In the first case, $v^\perp\cong E_7\oplus \langle 1\rangle^{k}$. In the second case, after an automorphism, $v=y=e_1+e_2$, and so $v^\perp\cong E_8\oplus \langle 2\rangle\oplus \langle 1\rangle^{k-2}$. However, $\delta(E_8\oplus\langle 2 \rangle\oplus\langle 1 \rangle^{k-2}) = \delta(E_8)+\delta(\langle 2 \rangle)=8+1=9$, and according to~\eqref{eq:dinv23} we have $\delta(S_2^3(T_{2,3}))=7$, in contradiction to inequality~\eqref{eq:delta}. Next suppose $n\geqslant 3$ and $M'\neq\langle n-1\rangle$; thus $n\leqslant 8$. By~\eqref{eq:twocases}, either $y=0$ or $y=(n-1)e_1$. If $y=0$, then $v=x\in \mathscr{T}_{n-1}$ and $v^\perp\cong \mathscr{T}_{n}\oplus\langle1\rangle^k$ by Lemma~\ref{lemma:distinguished}.

It remains to rule out the cases in which $y=(n-1)e_1$. Here $x^2=n-1$. As $M'=\mathscr{T}_{n-1}$ is even, we must have $n$ odd. Furthermore, the reduced part of $v^\perp$ has rank $10-n$, and its root lattice $\mathsf{R}(v^\perp)=\mathsf{R}(x^\perp)\oplus\mathsf{R}(y^\perp)=\mathsf{R}(x^\perp)$ is not of full rank (recall that, by definition, the root lattice only sees the minimal part of the lattice).
According to Table~\ref{fig:tableroots}, there are no such lattices of determinant $n$ unless $n=7$ and $\mathsf{R}(v^\perp)=A_2$. This would require $\mathsf{R}(x^\perp)=A_2$, where $x\in\mathscr{T}_6=A_1\oplus A_2$; consequently $x\in A_1$, contradicting the condition that $x^2=6$.
\end{proof}

\begin{proof}[Proof of Theorem~\ref{thm:2,5} (i)]
The case $n=1$ is implied by Theorem~\ref{thm:instanton}. We continue the notation from the proof of Theorem~\ref{thm:trefoil}(i), so that the lattice under consideration is $v^\perp$ where $v = (x,y)\in M'\oplus \langle 1\rangle^k=M$ with $v^2=n(n-1)$, $\det(v^\perp)=n$, and $\det(M)=n-1$.

\textbf{Case $n=2$}.  By the case $n=1$ we have that $M'$ is either zero, $E_8$ or $\Gamma_{12}$. If $M'=0$, then $v=y=e_1+e_2$ and $v^\perp\cong \langle 2 \rangle \oplus \langle 1 \rangle^{k-1}$. If $M'=E_8$, then either $y=0$ and $x^2=2$, in which case $v^\perp\cong E_7\oplus\langle 1 \rangle^k$, or $x=0$ and $y=e_1+e_2$, in which case $v^\perp=E_8\oplus \langle 2 \rangle\oplus \langle 1 \rangle^{k-2}$. Lemma~\ref{lemma:e821} below shows that we must have $k-2\geqslant 1$ here, as required by the statement of the theorem. If $M'=\Gamma_{12}$, then either $y=0$ and $x^2=2$ or $v=y=e_1+e_2$. In the first case, $v^\perp \cong \mathscr{C}_{2}\oplus \langle 1 \rangle^k$ by Lemma~\ref{lemma:clattices}. If instead $v=e_1+e_2$, then $v^\perp\cong \Gamma_{12}\oplus \langle 2\rangle\oplus \langle 1 \rangle^{k-2}$. This completes the case $n=2$.

\textbf{Case $n=3$}. By the case of $n=2$, $M'$ is among $\langle 2 \rangle$, $E_7$, $E_8\oplus \langle 2 \rangle$, $\Gamma_{12}\oplus \langle 2\rangle$,  $\mathscr{C}_2$. The cases $\langle 2 \rangle$ and $E_7$ are handled as in the proof of Theorem~\ref{thm:trefoil} (i), and yield the possibilities $\langle 3 \rangle\oplus\langle 1 \rangle^{k-1}$ and $E_6\oplus\langle 1 \rangle^{k}$. Suppose $M'=E_8\oplus \langle 2\rangle$. Write $x=(x_1,x_2)$ where $x_1\in E_8$ and $x_2\in \langle 2\rangle$. First suppose $y=0$. Then $x^2=x_1^2+x_2^2=6$. Suppose $x_2=0$. Then $v^\perp$ contains an orthogonal copy of $\langle 2\rangle$ and so cannot have determinant 3. The only other possibility is that $x_2$ generates $\langle 2 \rangle$. Then $x_1\in E_8$ is of square 4, and hence it is primitive; therefore, there exists $w\in E_8$ with $w\cdot x_1=1$, contradicting $d=2$ from~\eqref{eq:gcd}. Thus $y\neq 0$.  By~\eqref{eq:twocases} we may suppose $y=2e_1$. Then either $x=x_1\in E_8$ is a root or $x=x_2$ generates $\langle 2\rangle$. In the first case, there is again some $w\in E_8$ such that $w\cdot x_1=1$, contradicting~\eqref{eq:gcd}. In the latter case, $v^\perp\cong E_8\oplus\langle 3\rangle\oplus \langle 1 \rangle^{k-1}$.

Essentially the same argument shows that if $M'=\Gamma_{12}\oplus \langle 2\rangle$ then $v^\perp\cong \Gamma_{12}\oplus \langle 3 \rangle\oplus \langle 1 \rangle^{k-1}$. Finally, suppose $M'=\mathscr{C}_2$. If $y=0$ then $v^\perp\cong\mathscr{C}_3\oplus\langle 1 \rangle^k$ by Lemma~\ref{lemma:clattices}. If $y=2e_1$ then $x$ is a root in $\mathscr{C}_2$, and as argued above, we obtain that $d=1$, a contradiction to~\eqref{eq:gcd}. This completes the case $n=3$.

\textbf{Case $n=4$}. By the case of $n=3$, $M'$ is among $\langle 3 \rangle$, $E_6$, $E_8\oplus \langle 3 \rangle$, $\Gamma_{12}\oplus \langle 3\rangle$,  $\mathscr{C}_3$. The cases $\langle 3 \rangle$ and $E_6$ are handled as in the proof of Theorem~\ref{thm:trefoil} (i), and yield the possibilities $\langle 4 \rangle\oplus\langle 1 \rangle^{k-1}$ and $D_5\oplus\langle 1 \rangle^{k}$.

Suppose $M'=E_8\oplus \langle 3\rangle$. Write $x=(x_1,x_2)$ where $x_1\in E_8$ and $x_2\in \langle 3\rangle$. Suppose $y=0$. Then $x^2=x_1^2+x_2^2=12$. We may assume $x_2\neq 0$. Then either $x_2$ generates $\langle 3 \rangle$ or $x=x_2$ is twice a generator for $\langle 3\rangle$. However, as $x_1\in E_8$, and $x^2=x_1^2+x_2^2$, the former case contradicts that $x_1^2$ is even; the latter case implies $v^\perp$ is unimodular. Thus we must have $y\neq 0$. By~\eqref{eq:twocases} we may assume $y=3e_1$. Then $x^2=3$, and $x=x_2$ must generate $\langle 3 \rangle$. In this case $v^\perp\cong E_8\oplus \langle 4\rangle \oplus \langle 1 \rangle^{k-1}$. However,
\[
	\delta(E_8\oplus \langle 4 \rangle \oplus \langle 1 \rangle^{k-1} ) = \delta(E_8) + \delta(\langle 4 \rangle)  =8+1=9 > \delta(S_4^3(K))
\]
for $K\in\{U,T_{2,3},T_{2,5}\}$, in contradiction to~\eqref{eq:delta}. In summary, we cannot have $M'=E_8\oplus\langle 3\rangle$.

 Next, suppose $M'=\Gamma_{12}\oplus \langle 3\rangle$. The argument to rule this case out is much the same as for $M'=E_8\oplus\langle 3\rangle$. The only difference is that we must rule out the possibility that $y=0$ and $x_2^2=3$ in a different way, as $\Gamma_{12}$ is not an even lattice. Note here that $x_1^2=x^2-x_2^2=12-3=9$.
All vectors in $\Gamma_{12}$ of squared norm $9$ are primitive, hence there exists $w\in\Gamma_{12}$ such that $w\cdot x_1=w\cdot x = 1$, contradicting $d=3$ from~\eqref{eq:gcd}.

 Next suppose $M'=\mathscr{C}_3$. If $y=0$ then $v^\perp\cong D_9\oplus \langle 1 \rangle^{k}$ or $v^\perp\cong\mathscr{C}_4\oplus\langle 1 \rangle^k$ by Lemma~\ref{lemma:clattices}. However, in the case that $v^\perp\cong D_9\oplus \langle 1 \rangle^{k}$, we have
\[
	9 \; = \; \text{rk}(D_9) \; = \; \delta(D_9\oplus\langle 1 \rangle^{k}) \; > \; \delta(S_4^3(K))
\]
for $K\in\{U,T_{2,3},T_{2,5}\}$, in contradiction to~\eqref{eq:delta} and~\eqref{eq:dinv25}. Thus $D_9\oplus \langle 1 \rangle^{k}$ cannot occur. Next suppose $y\neq 0$. Then by~\eqref{eq:twocases} we may assume $y=3e_1$, so that $x = x_1\in\mathscr{C}_{3}$ has square 3. Now, if there exists $w\in \mathscr{C}_3$ such that $w\cdot x=1$, then we contradict $d=3$ from~\eqref{eq:gcd}. On the other hand, if there is no such $w$, then by Lemma~\ref{lemma:littlelattice}, the orthogonal complement $x_1^\perp\subset \mathscr{C}_3$ is unimodular. However, $\mathscr{C}_n$ does not have any unimodular summands when $n>1$, for it is a sublattice of $\Gamma_{12}$, an indecomposable unimodular lattice. Thus we cannot have $y\neq 0$. This completes the case $n=4$.

\textbf{Case $n\geqslant 5$}. In each case, $M'$ is among $\langle n-1\rangle$, $\mathscr{T}_{n-1}$, $\mathscr{C}_{n-1}$. The first two cases are dealt with as in the proof of Theorem~\ref{thm:trefoil} (i). The third case is dealt with as was the case $M'=\mathscr{C}_3$ when $n=4$ above. That is, we rule out the possibility of $y\neq 0$ using~\eqref{eq:gcd} and the fact that $\mathscr{C}_n$ does not have any unimodular summands when $n>1$; and when $y=0$, we have $v^\perp \cong \mathscr{C}_{n}\oplus\langle 1 \rangle^k$ by Lemma~\ref{lemma:clattices}.
\end{proof}

\begin{lemma}
Let $L$ be an odd lattice, and $v\in L^*$ a primitive covector.
Then $v^\perp \subset L$ is even if and only if $v$ is characteristic.\label{lemma:evencomp}
\end{lemma}

Both directions are probably well-known in the lattice community.
The `if' direction was also observed in the proof of~\cite[Proposition~6.2]{BG}.

\begin{proof}
We first prove that if $v$ is characteristic, then $v^\perp$ is even.
Indeed, if $w\in v^\perp$, then
\[
w\cdot w \equiv v\cdot w = 0 \pmod 2.
\]

We now prove the converse.
Since $v$ is primitive, there exists $z$ in $L$ such that $v\cdot z = 1$.
We claim that $v^\perp$ and $z$ span $L$:
indeed, for any $w\in L$, we can write $w = (w - (v\cdot w)z) + (v\cdot w)z$, and the first summand satisfies $v\cdot (w - (v\cdot w)z) = v\cdot w - (v\cdot w)(v\cdot z) = 0$, hence lives in $v^\perp$.
By assumption, $v^\perp$ is even and $L$ is odd; since $v^\perp$ and $z$ span $L$, $z\cdot z$ must be odd, hence
\[
z\cdot z \equiv  v\cdot z = 1 \pmod 2;
\]
on the other hand, for each $w\in v^\perp$ we have $v\cdot w = 0 \equiv w\cdot w \pmod 2$. That is, $v\cdot u \equiv u\cdot u \pmod 2$ for all $u$ in $v^\perp$ and $u = z$, and since $L$ is spanned by $v^\perp$ and $z$, we obtain that $v$ is characteristic.
\end{proof}

\begin{remark}\label{rmk:oddcoefficients}
If $v$ is characteristic and primitive in the odd, indefinite, unimodular lattice ${\rm I}_{a,b}$, then, with respect to a diagonal basis $h_1,\dots,h_a,e_1,\dots,e_{b}$, all coefficients of $v$ are odd.
Indeed, in general, ${\rm Char}(L)$ has a transitive action by $2L$, and it is immediate to check that $h_1+\dots+h_a+e_1+\dots+e_{b}$ is characteristic for ${\rm I}_{a,b}$.
\end{remark}

\begin{lemma}\label{lemma:e821}
Let $K$ be a knot in $S^3$.
Then $S^3_2(K)$ does not bound $E_8\oplus \langle 2 \rangle$.
\end{lemma}

\begin{proof}
Suppose the contrary, and let $W$ be a filling of $Y$, with torsion-free homology and intersection form $E_8\oplus \langle 2 \rangle$.
Let $X_2(K)$ denote the 4-manifold with boundary $S^3_2(K)$ obtained by attaching a 2-handle to $S^3=\partial B^4$ along $K$ with framing $+2$. Let $X = -W \cup X_2(K)$. Since $H_1(W)$ and $H_1(X_2(K))$ are torsion-free, so is $H_1(X)$;
then $X$ is spin if and only if $L_X$ is even.
But $L_X$ has signature $-8$, so $X$ cannot be spin by Rokhlin's theorem.
It follows that $L_X = {\rm I}_{1,9}$.

The generator of $H_2(X_2(K))$ is sent to an element $x\in L_X$ of square $2$, and $L_W= x^\perp \subset L_X$, since $\det L_W = 2$ (see Lemma~\ref{lemma:littlelattice}).
As $L_W$ is even, by Lemma~\ref{lemma:evencomp} and Remark~\ref{rmk:oddcoefficients}, all coordinates of $x$ must be odd, i.e. $x = ah-(b_1e_1+\dots+b_9e_9)$ with $a,b_1,\dots,b_9$ odd;
but then we obtain a contradiction, as $2 = x^2= a^2 -\sum b_i^2 \equiv 1-9 \equiv 0 \pmod 8$.
\end{proof}


\section{Constructions}\label{sec:constructions}

In this section we provide the constructions for parts (ii) of Theorems \ref{thm:trefoil} and \ref{thm:2,5}.
We start by excluding the trivial lattice: this follows either from Donaldson's theorem (because $E_8$ does not embed in a definite diagonal lattice) or from a $d$-invariants computation (because $d(S^3_{1}(T_{2,3})) = d(S^3_{1}(T_{2,5})) = -2$, while the correction term vanishes on 3-manifolds that are null-cobordant).
Note that for $n$-surgery $S_n^3(K)$ on a knot $K$, there is a simply-connected smooth 4-manifold $X_n(K)$ filling it with intersection form $\langle n \rangle$, obtained by attaching an $n$-framed 2-handle to the boundary of a 4-ball.
Upon connect-summing with copies of $\cp$, we obtain the lattices $\langle n\rangle \oplus \langle 1 \rangle^k$ listed in Theorems~\ref{thm:trefoil} and~\ref{thm:2,5}. Next, we recall some standard constructions that prove Theorem~\ref{thm:trefoil} (ii).

Given a knot $K$, we will write $mK$ as a shorthand for the connected sum of $m$ copies of $K$.

\subsection{Seifert spaces and binary polyhedral spaces}

\begin{table}
\centering
\begin{tabular}{ l | c | c | c | l }
 Surgery & Other name & Seifert invariants & Binary Polyhedral &  Lattice \\
  \hline
&&&&\\
$S_1(T_{2,3})$ & $-\Sigma(2,3,5)$ & $(2; \frac{1}{2} , \frac{2}{3},\frac{4}{5})$ & $-SU(2)/I^\ast$  & $\mathscr{T}_1=E_8$\\
		&&&&\\
$S_2^3(T_{2,3})$ & $-\Sigma(2,3,4)$ &  $(2; \frac{1}{2} , \frac{2}{3},\frac{3}{4})$ & $-SU(2)/O^\ast$   & $\mathscr{T}_2=E_7$\\
&&&&\\
$S^3_3(T_{2,3})$  & $-\Sigma(2,3,3)$ & $(2; \frac{1}{2} , \frac{2}{3},\frac{2}{3})$ & $-SU(2)/T^\ast$ & $\mathscr{T}_3=E_6$\\
&&&&\\
$S^3_4(T_{2,3})$  & $P(3,1)$ & $(2; \frac{1}{2} , \frac{1}{2},\frac{2}{3})$ & $-SU(2)/D_{12}^\ast$ & $\mathscr{T}_4=D_5$\\
&&&&\\
$S_5^3(T_{2,3})$  & $L(5,1)$ & $(2; \frac{1}{2} , \frac{2}{3})$  & $-SU(2)/\Z_5$ & $\mathscr{T}_5=A_4$ \\
&&&&\\
$S_6^3(T_{2,3})$ & $L(2,1)\#L(3,1)$ &  & & $\mathscr{T}_6=A_1\oplus A_2$ \\
&&&&\\
$S_7^3(T_{2,3})$ & $L(7,5)$ & $(-2;-\frac{1}{2},-\frac{1}{3})$ & & $\mathscr{T}_7=\Lambda(2,4)$ \\
&&&&  \\
  \hline  
\end{tabular}
\caption{The first several positive integer surgeries on the right-handed trefoil, with the distinguished positive definite lattices $\mathscr{T}_n$ that they bound.}\label{fig:table1}
\end{table}

The $n$-surgeries of $T_{2,3}$ for $1\leqslant n \leqslant 7$ are listed in Table~\ref{fig:table1}. The Seifert-fibered descriptions of these manifolds are from Moser~\cite{moser}, who showed that $(p/q)$-surgery on an $(r,s)$-torus knot is $-L(r,s)\#{-L(s,r)}$ if $p/q=rs$, $-L(|p|,qs^2)$ if $|p-qrs|=1$, and is otherwise a Seifert-fibered space with at most three singular fibers and base orbifold the 2-sphere.
(For us, $L(p,q)$ is $(-p/q)$-surgery along the unknot.)
For $n\neq 6$,
\[
    S_n^3(T_{2,3}) \; = \; M\left(-1;-\frac{1}{2},-\frac{1}{3},\frac{1}{n-6}\right),
\]
the right-hand side denoting the Seifert-fibered space with Seifert invariants $\mathbf{b}=(b;b_1/a_1,\ldots,b_k/a_k)$. Here $b\in\Z$ and $b_i/a_i$ are reduced fractions with $a_i>0$. The homeomorphism class of a Seifert-fibered space is classified by the Euler number $e=b-\sum b_i/a_i$ and the reductions of $b_i/a_i$ modulo 1. Thus $(-1;-1/2,-1/3,-1/5)$ and $(2;1/2,2/3,4/5)$ both determine $S_1^3(T_{2,3})$. Further, $M(b;b_1/a_1,\ldots,b_k/a_k)$ bounds a plumbed 4-manifold $X_\mathbf{b}$ as described in Figure~\ref{fig:seifertdiagram} for $k=3$, which also determines a surgery diagram. Here the integers $t_{i,j}$ come from the Hirzebruch--Jung continued fraction for $a_i/b_i$:
\begin{equation*}
\frac{a_i}{b_i} = t_{i,1}-\cfrac{1}{t_{i,2}-\cfrac{1}{\cdots-\frac{1}{t_{i,m_i}}}}
\end{equation*}
The lattices appearing in the right-hand column of Table~\ref{fig:table1} are realized by the plumbings $X_\mathbf{b}$ where $\mathbf{b}$ is given in the corresponding row of Table~\ref{fig:table1}. Upon connect summing these examples with copies of $\cp$ we obtain all lattices listed in Theorem~\ref{thm:trefoil}. This completes the proof of Theorem~\ref{thm:trefoil}.

Six of the seven surgeries in Table~\ref{fig:table1} are distinguished for admitting spherical geometry. In fact, each is realized, after possibly reversing orientation, as a binary polyhedral space, i.e. a quotient of $SU(2)$ by a finite subgroup $\Gamma$, see e.g.~\cite[Section~1.2.1]{saveliev}. Each such space is the boundary of a minimal resolution of a Kleinian singularity $\C^2/\Gamma$ and its intersection form is a negative definite root lattice. For $n\leqslant 5$, each such resolution is orientation-reversing diffeomorphic to the corresponding plumbing from above.

Similar to the Seifert descriptions for surgeries on $T_{2,3}$, Moser's results imply that for $1\leqslant n\leqslant 9$, $S_n^3(T_{2,5})= M(-1; -1/2,-2/5,1/(n-10))$, and the plumbings  $X_\mathbf{b}$ with $\mathbf{b}=(2;1/2,3/5,(9-n)/(10-n))$ realize the lattices $\mathscr{C}_n$. For $n=10$, we have $S_{10}^3(T_{2,5})=L(2,1)\# L(5,3)$; as $A_1$ fills $L(2,1)$ and $\Lambda(3,2)$ fills $L(5,3)$, the lattice $\mathscr{C}_{10}$, defined as the direct sum thereof, fills $S_{10}^3(T_{2,5})$. Finally, for $n=11$ we have $S_{11}^3(T_{2,5})=L(11,7)$, and $\Lambda(3,4)$ fills this lens space. After connect-summing these examples with copies of $\cp$, we obtain all the lattices $\mathscr{C}_n\oplus \langle 1 \rangle^k$ listed in Theorem~\ref{thm:2,5}.

\begin{figure}
\centering
\begin{tikzpicture}[scale=0.8]
	\draw[thick,rounded corners=6pt]
(-1,0) -- (0,0) -- (0,2.25) -- (-6,2.25) -- (-6,0) -- (-1,0);
\draw[white,line width=1mm] (-4.45,0) -- (-4.05,0);
\draw[white,line width=1mm] (-.95,0) -- (-0.55,0);
\draw[white,line width=1mm] (-2.3,0) -- (-2.7,0);
\draw[thick,rounded corners=6pt]
(-5.25,0.2) -- (-5.25,1) -- (-4.25,1) -- (-4.25,-1) -- (-5.25,-1) -- (-5.25,-.2);
\draw[thick,rounded corners=6pt]
(-1.75,0.2) -- (-1.75,1) -- (-.75,1) -- (-.75,-1) -- (-1.75,-1) -- (-1.75,-.2);
\draw[thick,rounded corners=6pt]
(-3.5,0.2) -- (-3.5,1) -- (-2.5,1) -- (-2.5,-1) -- (-3.5,-1) -- (-3.5,-.2);
\node[font=\small] at (-6.35,1.25) {$b$};
\node[font=\small] at (-4.75,1.5) {$a_1/b_1$};
\node[font=\small] at (-1.25,1.5) {$a_3/b_3$};
\node[font=\small] at (-3,1.5) {$a_3/b_3$};
\end{tikzpicture}\qquad
\begin{tikzpicture}
	\draw (0,0) -- (1,1);
	\draw (0,0) -- (1,0);
	\draw (0,0) -- (1,-1);
	\draw (1,1) -- (2.75,1);
	\draw (1,0) -- (2.75,0);
	\draw (1,-1) -- (2.75,-1);
	\draw (4.25,1) -- (5,1);
	\draw (4.25,0) -- (5,0);
	\draw (4.25,-1) -- (5,-1);
	
	\draw[fill=black] (0,0) circle(.06);
	\draw[fill=black] (1,0) circle(.06);
	\draw[fill=black] (1,1) circle(.06);
	\draw[fill=black] (1,-1) circle(.06);
	\draw[fill=black] (2,0) circle(.06);
	\draw[fill=black] (2,1) circle(.06);
	\draw[fill=black] (2,-1) circle(.06);
	\draw[fill=black] (5,0) circle(.06);
	\draw[fill=black] (5,1) circle(.06);
	\draw[fill=black] (5,-1) circle(.06);
	
	\node[font=\small] at (-.45,0) {$b$};
	\node[font=\small] at (1,1.4) {$t_{1,1}$};
	\node[font=\small] at (2,1.4) {$t_{1,2}$};
	\node at (3.5,1) {$\ldots$};
	\node[font=\small] at (5,1.4) {$t_{1,m_1}$};
	\node[font=\small] at (1,0.4) {$t_{2,1}$};
	\node[font=\small] at (2,0.4) {$t_{2,2}$};
	\node at (3.5,0) {$\ldots$};
	\node[font=\small] at (5,0.4) {$t_{2,m_2}$};
	\node[font=\small] at (1,-.6) {$t_{3,1}$};
	\node[font=\small] at (2,-.6) {$t_{3,2}$};
	\node at (3.5,-1) {$\ldots$};
	\node[font=\small] at (5,-.6) {$t_{3,m_3}$};
\end{tikzpicture}
\caption{}\label{fig:seifertdiagram}
\end{figure}

\subsection{PL spheres and rational complex curves}

In the next two sections, we will study how fillings of 3-manifolds obtained by doing surgery along a knot in $S^3$ can be constructed starting from embedded PL spheres.

Consider an embedded PL sphere $S$ in a closed, oriented, 4-manifold $X$.
The surface $S$ is smooth away from a finite number of points $p_1,\ldots,p_m$; at a neighborhood of $p_i$, $S$ is the cone over a knot $K_i\subset S^3$.
We say that the singularity at $p_i$ is of {\emph{type $K_i$}}, and that $K_i$ is the \emph{link} of $S$ at $p_i$.
Call $n = [S]\cdot [S]$ the self-intersection of $S$;
then it is a good exercise to show that a regular (closed) neighborhood $N$ of $S$ is diffeomorphic to $X_n(K)$, the trace of $n$-surgery along the knot $K := K_1\#\dots\# K_m$. In particular, $W_S = -(X\setminus{\rm Int}(N))$ has boundary $S^3_n(K)$.
Furthermore, if $n>0$, then $W_S$ is positive definite.
With a conscious abuse of terminology, we will refer to $W_S$ as the \emph{complement} of $S$.

We will study the topology of complements of PL spheres in the next section; here we pause instead for a short trip in complex algebraic geometry, which is an excellent source of such objects.
Classical references are, for instance,~\cite{Wall} for general treatment of curve singularities, and ~\cite{namba} for the study of plane singular curves; we also suggest~\cite{moe} for a more modern, and more topologically flavored, exposition of some of the material from~\cite{namba}, and for a quicker introduction to singularities of complex curves.

An irreducible complex curve $C$ in a $\cp$ is defined by an equation $F(x,y,z) = 0$, where $F$ is an irreducible homogeneous polynomial;
the degree of $F$ is the degree of the curve.
We say that $C$ is rational and cuspidal if $C$ is an embedded PL sphere in $\cp$.
For instance, if $a<b$ are coprime positive integers, the curve defined by the equation $x^az^{b-a} - y^b = 0$ is rational and cuspidal;
the link at $(0:0:0)$ is the torus knot $T_{a,b}$.

Each singular point $p$ comes with a multiplicity, that counts the number of local intersections of $C$ with a line;
the multiplicity of the singularity of type $T_{a,b}$ (with $a<b$) is $a$.
We can interpret the singularity in terms of blow-ups: when we blow up $\cp$ at a singular point of a degree-$d$ curve $C$ that has multiplicity $a$, the homology class of the proper transform of $C$ in the blow-up is $dh - ae$, where $h$ is the line class in $H_2(\cp)$ and $e$ is the homology class of the exceptional divisor in the blow-up.
Similarly, if we blow up at a $k$-fold point (e.g. a transverse double point), the homology class of the proper transform gains a summand $-ke$.

Finally, recall that every curve singularity can be resolved by blow-ups;
the collection of multiplicities of the (non-trivial) singularities encountered in the process forms the \emph{multiplicity sequence} of the singularity;
for example, a singularity of type $T_{3,5}$ has multiplicity sequence $[3,2]$: the multiplicity of $T_{3,5}$ is $5$, and blowing up yields a curve with singularity $T_{2,3}$;
the latter has multiplicity $2$, and a single blow-up resolves it.
Analogously, one sees that the multiplicity sequence of the singularity of type $T_{a,ka+1}$ is the string $[a,\dots,a]$ of length $k$.

\subsection{Fillings constructed from PL spheres}

As mentioned at the beginning of the previous section, we will look here at an alternative perspective one can take on the trefoil lattices $\mathscr{T}_n$, which shows that they also fill surgeries along the cinquefoil knot $T_{2,5}$.
Indeed, we describe here a more general framework to realize lattices as fillings by means of PL spheres.

We focus here on PL spheres in closed, oriented 4-manifolds $X$ with $H_1(X) = 0$, $b^+(X) = 1$, and odd intersection form; for instance, a blow-up of $\cp$ gives an example of such a manifold.
Recall that, if $S$ is such a sphere, we called $W_S$ the complement of (a small open, regular neighborhood of) $S$ in $X$.

Clearly, the intersection form $L_S$ of $W_S$ is the orthogonal of $[S]$ in $H_2(X)$. Furthermore,
by Lemma~\ref{lemma:littlelattice}, if $[S]$ is primitive, $\det L_S = n$, and, using excision for the pairs $(X,N)$ and $(W_S,\partial W_S)$ and the long exact sequence for a pair, one sees that $W_S$ has torsion-free homology.
(Note that in general $\det L_S$ divides $n$.)

Vice-versa, fix $n>0$ and a knot $K\subset S^3$. If $W$ is a positive definite filling of $S^3_n(K)$ with torsion-free homology, then $X = X_n(K) \cup -W$ is a closed 4-manifold with torsion-free homology, which contains a PL embedded sphere of square $n$, whose unique singularity is a cone over $K$.
(If $K$ is a connected sum, one can split the singularity into singularities that are cones over the connected summands of $K$.)
By surgery along loops in $X$, we can also ensure that $H_1(X) = 0$.
By construction, $b^+(X) = 1$, and by Donaldson's Theorem B of~\cite{d-connections,donaldson-orientations}, $X$ is not spin.
Since $H_1(X) = 0$, this implies that the intersection form of $X$ is odd, and hence diagonal.

In the following statement, write $h$ for a generator of $H_2(X)$, where $X$ is any homology $\cp$.


\begin{prop}\label{prop:PLspheres}
Let $K$ be either $T_{2,3}$, $T_{2,5}$, or $T_{3,4}$.
Then there is a PL sphere $S$ in a homotopy $\cp$ in the homology class $3h$ with a unique singularity of type $K$.
\end{prop}

\begin{proof}
When $K = T_{2,3}$, we may take $S$ to be the cuspidal cubic in $\cp$, i.e. the zero set of the polynomial $x^2z - y^3$.

When $K = T_{2,5}$, $S^3_9(K) = L(9,4)$; $L(9,4)$ is obtained as surgery along a knot in $S^1\times S^2 = \partial(S^1\times D^3)$~\cite[Theorem 1.3]{LecuonaBaker}, so it bounds a rational homology ball $W$ constructed with one 1-handle and 2-handle;
gluing $X_9(K)$ and $-W$ along their common boundary yields $X$, a homotopy $\cp$; indeed, $X$ has a handle decomposition with no 1-handles, and $\chi(X) = 3$.

When $K = T_{3,4}$, $S^3_9(K)$ bounds the rational homology ball shown in Figure~\ref{f:T349}, and the same argument as above shows that $X_9(K)$ embeds in a homotopy $\cp$.
\end{proof}

\begin{figure}
\begin{center}
\labellist
\pinlabel $-1$ at 210 68
\endlabellist
\includegraphics[scale=0.8]{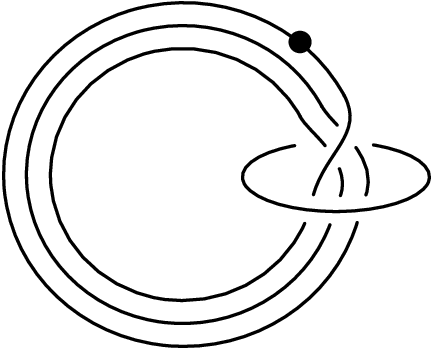}
\end{center}
\caption{A rational homology ball filling $S^3_9(T_{3,4})$.
Adding a $-1$-framed meridian to the dotted circle gives an embedding of $-X_9(T_{3,4})$ in $\overline{\cp}$.}\label{f:T349}
\end{figure}

\begin{remark}
In the case when $K = T_{2,5}$ and $K = T_{3,4}$ one can show using Kirby calculus that we can choose the gluing diffeomorphism in such a way as to obtain $\cp$.
\end{remark}

Given $X$ and $S$ as in Proposition~\ref{prop:PLspheres}, we can produce primitive homology classes in blow-ups of $X$, by blowing up along points of $S$.
Recall that we write $h$ for the generator of $H_2(X)$. When blowing up, we will write $e_1,e_2,\dots$ for the classes of exceptional divisors. The lattice of $X$ blown up $\ell$ times is then isomorphic to $\text{I}_{1,\ell}$, the unimodular odd lattice of rank $1+\ell$ and signature $1-\ell$, with diagonalizing basis $h,e_1,\ldots,e_\ell$ such that $h^2=1$ and $e_i^2=-1$. Henceforth this isomorphism between $\text{I}_{1,\ell}$ and the lattice of $\smash{X\# \ell \cpbar}$ will be implicit.

Let $ 1\leqslant n\leqslant 8$. We realize a PL sphere $S_n$ in a $(9-n)$-fold blow-up $X$ by blowing up at $9-n$ generic points along the surface $S$ provided by Proposition~\ref{prop:PLspheres}. The PL sphere $S_n$ has the same type of singularity as $S$. The primitive homology class $[S_n]=3h - e_1 - \dots - e_{9-n} \in \text{I}_{1,9-n}$ has $[S_n] \cdot [S_n] = n > 0$, and in particular its complement is negative definite. For the following statement, we recall our convention from Section \ref{sec:lattices} that $\mathscr{T}_8=\langle 8\rangle$.

\begin{lemma}\label{lemma:Tnorthogonal}
For $1 \leqslant n \leqslant 8$, the orthogonal complement of $[S_n]\in {\emph{I}}_{1,9-n}$ is isomorphic to $-\mathscr{T}_n$.
\end{lemma}

\begin{proof}
When $1 \leqslant n\leqslant 6$, $[S_n]^\perp$ is spanned by $e_1-e_2, \dots, e_{8-n}-e_{9-n}$ and $h-e_1-e_2-e_3$, which generate the root lattice $-\mathscr{T}_n$. When $n=7$, it is spanned by $h-2e_1-e_2$ and $e_1-e_2$, which generate the lattice $-\Lambda(2,4)$. Finally, when $n=8$,  $[S_n]^\perp$ is spanned by $h-3e_1$, which has square $-8$.
\end{proof}

\begin{corollary}
For $1\leqslant n \leqslant 8$, $\mathscr{T}_n$ fills $S^3_n(K)$ for $K = T_{2,3}, T_{2,5}, T_{3,4}$.
\end{corollary}

After connect-summing the realizations in this corollary with copies of $\cp$, we obtain the lattices $\mathscr{T}_n\oplus \langle 1\rangle^k$ listed in Theorem~\ref{thm:2,5}. We have also recovered all non-diagonal lattices in Theorem~\ref{thm:trefoil}.

We now turn to realising $E_8\oplus\langle 2\rangle\oplus\langle 1\rangle$ and $E_8\oplus\langle 3\rangle$ as intersection forms of fillings of $S^3_2(K)$ and $S^3_3(K)$ respectively, for $K = T_{2,5}$. We consider a few other knots along the way. For the following statement, when $k\leqslant \ell$, we view $\text{I}_{1,k}$ as the sublattice of $\text{I}_{1,\ell}$ spanned by $h,e_1,\dots,e_k$.

\begin{lemma}\label{lemma:double}
Suppose $v\in\emph{\text{I}}_{1,k}$ with $v^2 = 1$, whose orthogonal complement is $\Lambda\subset \emph{\text{I}}_{1,k}$; then:
\begin{itemize}
\item[(i)] the orthogonal of $2v-e_{k+1}$ in $\emph{\text{I}}_{1,k+1}$ is isomorphic to $\Lambda \oplus \langle-3\rangle$;
\item[(ii)] the orthogonal of $2v-e_{k+1}-e_{k+2}$ in $\emph{\text{I}}_{1,k+2}$ is isomorphic to $\Lambda \oplus \langle-2\rangle\oplus \langle-1\rangle$.
\end{itemize}
\end{lemma}

\begin{proof}
We only prove (ii); (i) is analogous.
Pick a basis $w_1,\dots,w_k$ for $\Lambda$;
we claim that $w_1,\dots,w_k$, $e_{k+1}-e_{k+2}$, $v-e_{k+1}-e_{k+2}$ is a basis for the orthogonal of $2v-e_{k+1}-e_{k+2}$ exhibiting the orthogonal decomposition.
Indeed, it is easy to see that $w_j$, $e_{k+1}-e_{k+2}$, and $v-e_{k+1}-e_{k+2}$ are pairwise orthogonal for every $j$, that $(e_{k+1}-e_{k+2})^2 = -2$, and that $(v-e_{k+1}-e_{k+2})^2 = -1$, and hence they span $\Lambda \oplus \langle-2\rangle\oplus \langle-1\rangle$.
Since both the determinant of this lattice and the square of $2v-e_{k+1}-e_{k+2}$ are equal to 2, in fact the orthogonal to $2v-e_{k+1}-e_{k+2}$ is isomorphic to $\Lambda \oplus \langle-2\rangle\oplus \langle-1\rangle$, as claimed.
\end{proof}

In light of Lemmas~\ref{lemma:Tnorthogonal} and~\ref{lemma:double} above, in order to realize $E_8\oplus\langle 2\rangle\oplus\langle 1\rangle$ and $E_8\oplus\langle 3\rangle$ as fillings of $S^3_2(K)$ and $S^3_3(K)$, it suffices to find a homotopy $\smash{\cp \# 8\cpbar}$ and realize the homology class $6h-2e_1-\dots-2e_8 \in \text{I}_{1,8}$ in its lattice as the class of a PL sphere with a singularity of type $K$.

\begin{prop}\label{prop:E823}
Let $K$ be either $2T_{2,3}$, $T_{2,5}$, $T_{2,7}$, or $T_{3,4}$.
Then there is a PL sphere $S$ in a homotopy $\smash{\cp\#8\cpbar}$ in the homology class $6h-2e_1-\dots-2e_8$ with a unique singularity of type $K$.
\end{prop}

\begin{proof}
We start with $K = T_{2,5}$. There is a singular degree-4 complex curve $C$ in $\cp$ that is rational and has one singularity of type $T_{2,5}$ and one of type $T_{2,3}$, see e.g.~\cite[Theorem~2.2.5(2)]{namba}.
We add two generic lines $\ell, \ell'$ to $C$ and smooth two of the resulting double intersections, one on each line.
So far we have constructed an \emph{immersed} PL sphere in $\cp$ with the two singularities of $C$ and seven additional double points:
indeed, there are eight intersections points between $C$ and $\ell \cup \ell'$, and one intersection between $\ell$ and $\ell'$, but two of these intersections have been smoothed.
Blowing up at the remainaing double points and at the trefoil cusp yields a PL sphere whose only singularity is of type $T_{2,5}$, and whose homology class is the desired one.

Blowing up once at the $T_{2,5}$-singularity instead of at the $T_{2,3}$-singularity in the last step, we obtain a PL sphere with two singular points of type $T_{2,3}$. Thus the same result holds for $2T_{2,3}$.

When $K = T_{3,4}$, we observe that there is an immersed concordance from $K$ to $T_{3,5}$, given by a positive crossing change;
in Lemma~\ref{lemma:T354} we then exhibit a 3-band cobordism $\Sigma$ from $T_{3,5}$ to $T_{4,4}$ in $S^3\times I$;
if we view $S^3$ as the boundary of the disc bundle $E$ over $S^2$ with Euler number $+1$ (i.e. a neighbourhood of a line in $\cp$), we can cap off $T_{4,4}$ in $E$ with four disjoint disks, each intersecting the 0-section of $E$ transversely and positively once.
We can build an immersed PL sphere in $\cp$ by gluing the cone over $K$, the immersed concordance to $T_{3,5}$, $\Sigma$, and the four discs in $E$; this sphere lives in the homology class $4h$, since it meets a line (the 0-section in $E$) algebraically four times.
The same argument as above now concludes the proof.

The same argument works for $K = T_{2,7}$: there is a (positive-to-negative) crossing change to $T_{2,9}$, and we apply Lemma~\ref{lemma:T354} as above.
\end{proof}

\begin{lemma}\label{lemma:T354}
There are genus-$0$ cobordisms from $T_{3,5}$ and $T_{2,9}$ to $T_{4,4}$.
\end{lemma}

\begin{proof}
We are going to exhibit a ribbon cobordism from $T_{3,5}$ to $T_{4,4}$ obtained by attaching three (orientation-coherent) bands to $T_{3,5}$.
An Euler characteristic computation immediately shows that such a cobordism has genus 0.

We will write $x,y,z$ for the three standard generators of the 4-braid group $B_4$, which satisfy the relations $xyx = yxy, yzy = zyz, xz=zx$; we denote with $\xin$ ($\yin, \zin$) the inverse of $x$ ($y$, $z$, respectively).
We also call $\Delta$ the full twist on four strands, i.e. $\Delta = (xyz)^4$.
The knot $T_{3,5}$ is the closure of the 4-braid $(xy)^5z$ (which is a positive Markov stabilisation of the 3-braid $(xy)^5$), 
and $T_{4,4}$ is the closure of the braid $\Delta$.

Observe that we can write
\[
\Delta^{-1} = (\zin\yin\xin)^2(\zin\yin\xin\zin\yin\xin) = (\zin\yin\xin)^2(\zin\yin\zin\xin\yin\xin) = (\zin\yin\xin)^2(\yin\zin\yin\xin\yin\xin).
\]
We then have
\begin{align*}
(xy)^5z &= (xy)^5z\Delta^{-1}\Delta = (xy)^5z\zin\yin\xin\zin\yin\xin\yin\zin\yin\xin\yin\xin\Delta \sim xyxy\zin\yin\xin\yin\zin\Delta,
\end{align*}
where $\sim$ denotes conjugation in the braid group, and we have used that $\Delta$ is in the center of $B_4$.
We now attach three bands to this braid by cancelling three factors in the above expression; this corresponds to adding the inverses of the corresponding generator, which is indeed an oriented band:
\[
(xy)^5z \sim xyxy\zin\yin\xin\yin\zin\Delta \leadsto \cancel{x}yxy\cancel{\zin}\yin\xin\yin\cancel{\zin}\Delta = \Delta.
\]

For the cobordism from $T_{2,9}$, we view $T_{2,9}$ as the $(2,9)$-cable of the unknot, viewed as the $(2,1)$-torus knot.
Then attaching three bands as in Figure~\ref{f:T294}, we get the desired cobordism; see~\cite{Baader}.
In braid terms, this corresponds to viewing $T_{2,9}$ as the closure of the 4-braid $(xyz)^2(xz)^2x$ (with $x$ interchanging the two top strands in the figure, $y$ the two central ones, and $z$ the two bottom ones), and adding three generators (two $y$ and one $\xin$) as follows:
\[
(xyz)^2(xz)^2x \leadsto (xyz)^2(xyz)(xyz)x\xin = \Delta.\qedhere
\]
\end{proof}

\begin{figure}
\includegraphics[scale=1]{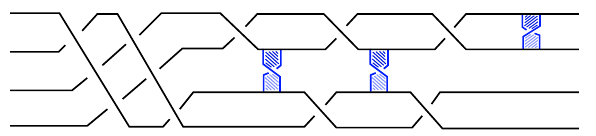}
\caption{The band attachments of Lemma~\ref{lemma:T354}.}\label{f:T294}
\end{figure}

\begin{corollary}
$E_8\oplus\langle 2\rangle\oplus \langle 1 \rangle$ fills $S_2^3(K)$ and $E_8\oplus \langle 3\rangle$ fills $S^3_3(K)$ for $K = 2T_{2,3}$, $T_{2,5}$, $T_{2,7}$, $T_{3,4}$.
\end{corollary}

We have now completed the proof of Theorem~\ref{thm:2,5} (ii), having realized all lattices listed therein for $S_n^3(T_{2,5})$. Nevertheless, we continue our discussion and show how to realize the lattices $\mathscr{C}_n$ in the same fashion as the others.

\begin{lemma}\label{lemma:Gamma12}
For $1\leqslant n \leqslant 11$, the orthogonal complement of the vector $v=4h-2e_1-e_2-\dots-e_{13-n}$ in ${\emph{\text{I}}}_{1,13-n}$ is isomorphic to $-\mathscr{C}_n$.
\end{lemma}

\begin{proof}
When $n\leqslant 10$, we can exhibit a basis of $v^\perp$ as follows: $e_2-e_3,\dots,e_{12-n}-e_{13-n}$, $2e_1-h$, and $h-e_1-e_2-e_3$.
When $n=11$, we can choose the basis $2e_1-h$ and $h-e_1-2e_2$.
One immediately verifies that these bases realize the defining plumbing graphs for $-\mathscr{C}_n$, as in shown in Figure~\ref{f:Cns}.
\end{proof}

\begin{prop}\label{prop:C-lattices}
Let $K$ be either $2T_{2,3}$, $T_{2,5}$, or $T_{3,4}$.
Then there is a PL sphere $S$ in a homotopy $\smash{\cp\#\cpbar}$ in the homology class $4h-2e_1$ with a unique singularity of type $K$.
\end{prop}

\begin{proof}
	This follows from the constructions provided in the proof of Proposition~\ref{prop:E823}. For example, for the case of $K=T_{2,5}$, we omit the addition of the two additional lines $\ell$ and $\ell'$, and only blow up at the trefoil cusp singularity.
\end{proof}

\begin{corollary}
For $1\leqslant n \leqslant 11$, $\mathscr{C}_n$ fills $S_n^3(K)$ for $K = 2T_{2,3}$, $T_{2,5}$, $T_{3,4}$.
\end{corollary}

\subsection{Further examples}

The above constructions concern the lattices appearing in Theorems \ref{thm:trefoil} and \ref{thm:2,5}, but the method is clearly very general. Here we offer a few more examples involving other lattices, including some related to the discussion in~\cite[Section~5]{scaduto-forms}, by looking at surgeries along knots of slice genus at least 3.

If we go by rank, the first reduced unimodular positive definite lattice that is neither isomorphic to $E_8$ nor $\Gamma_{12}$ is the rank-14 lattice $E_7^2$, which is labelled by its root lattice $E_7\oplus E_7$. This lattice is isomorphic to the complement of the vector $6h-2e_1-\cdots -2e_7-e_8-\cdots-e_{14}\in\text{I}_{1,14}$, see e.g.~\cite[Section~4]{scaduto-niemeier} for this and later such claims. This motivates us to define $\mathscr{E}_{n}$ for $1\leqslant n \leqslant 7$ as the orthogonal complement of the vector $6h-2e_1-\cdots -2e_7-e_8-\cdots-e_{15-n}\in \text{I}_{1,15-n}$. It is readily verified that $\mathscr{E}_n$ is a lattice of rank $15-n$, determinant $n$, and $\mathscr{E}_1=E_7^2$.

\begin{prop}
For $1\leqslant n\leqslant 7$, $\mathscr{E}_n$ fills $S_n^3(K)$ where $K=T_{2,3}\# T_{2,5}, 3T_{2,3}, T_{2,7}$, or $T_{3,4}$.\label{prop:e72}
\end{prop}

\begin{proof}
Begin with a rational quartic in $\cp$ with cuspidal singularities of types $K_1,\ldots,K_m$. Add a generic conic, so that there are 8 double points. Smooth one double point, blow up at the remaining 7 double points, and blow up at $8-n$ generic points. The resulting PL sphere is in the homology class $6h-2e_1-\cdots -2e_7-e_8-\cdots-e_{15-n}$. This shows that $\mathscr{E}_n$ fills $S_n^3(K_1\#\cdots \# K_m)$. Finally, a rational cuspidal quartic either has one singularity of type $T_{3,4}$ or $T_{2,7}$; two cusps of types $T_{2,5}$ and $T_{2,3}$; or three singularities each of type $T_{2,3}$. See~\cite[Theorem~2.2.5]{namba} and the discussion thereafter, or~\cite[Section~3.1]{moe}.
\end{proof}

\begin{remark}
A similar argument works for $K=T_{3,5}$, $T_{2,9}$, and $T_{4,5}$, since in both cases there is a genus-0 cobordism to $T_{4,4}$, which is exactly what we are using to construct the required homology class.
\end{remark}

The next reduced positive definite unimodular lattice, by rank, is the rank 15 lattice $A_{15}$, again labelled by its root lattice. This lattice is isomorphic to the orthogonal complement of $4h-e_1-\cdots-e_{15}\in\text{I}_{1,15}$. We thus define $\mathscr{A}_{n}$ for $1\leqslant n\leqslant 15$ to be the orthogonal complement of $4h-e_1-\cdots -e_{16-n}\in\text{I}_{16-n}$. Then $\mathscr{A}_n$ is a lattice of rank $16-n$, determinant $n$, and $\mathscr{A}_1=A_{15}$. Taking any rational cuspidal quartic as in the proof of Lemma \ref{prop:e72}, and blowing up $16-n$ generic times, we are led to the following proposition.

\begin{prop}
For $1\leqslant n\leqslant 15$, $\mathscr{A}_n$ fills $S_n^3(K)$ where $K=T_{3,4}$, $T_{2,7}$, $T_{2,3}\# T_{2,5}$, $3T_{2,3}$, $T_{2,9}$, $T_{3,5}$, or $T_{4,5}$
\end{prop}

Note that the last three cases above do not come from algebraic geometry; for instance, as in Figure~\ref{f:T349}, one can show that the trace of $16$-surgery along $T_{4,5}$ embeds in $\cp$, and the same argument outlined above applies in this case, too.

For the remaining two cases, we know from~\cite{AcetoGolla} that $16$-surgery along both knots bound rational homology ball, and in fact it is not hard to show that one can construct such rational balls using only handles of index at most 2 (since $S^3_{16}(T_{3,5})$ is a lens space, this follows from~\cite{LecuonaBaker}), so that gluing the rational homology ball with the trace of the surgery yields a homotopy $\cp$.

Finally, we consider the unique reduced positive definite lattice of rank 16 that is odd: $D_8^2$. This is isomorphic to the orthogonal complement of $8h-4e_1-3e_2-2e_3-\cdots -2e_{10}-e_{11}-\cdots -e_{16}\in \text{I}_{1,16}$. Thus we define $\mathscr{D}_n$ for $1\leqslant n \leqslant 7$ to be the orthogonal complement of $8h-4e_1-3e_2-2e_3-\cdots -2e_{10}-e_{11}-\cdots -e_{17-n}\in \text{I}_{1,17-n}$, a lattice of rank $17-n$, determinant $n$, and $\mathscr{D}_1=D_8^2$. We have:

\begin{prop}
For $1\leqslant n\leqslant 7$, $\mathscr{D}_n$ fills $S_n^3(K)$, where $K$ is among $T_{2,3}\# T_{3,4}$, $T_{3,5}$, $T_{2,9}$, $T_{2,3}\#T_{2,7}$, $2T_{2,5}$, $4 T_{2,3}$ and $2T_{2,3}\# T_{2,5}$.
\end{prop}

\begin{proof}
According to~\cite[Theorem~1.1]{fenske} Case 4 with $a=d=2$, there exists a rational septic in $\cp$ with two cuspidal singularities: one has multiplicity sequence $[4,2,2,2]$, and the other $[3,3]$. Add a line to this curve, introducing seven double points, and smooth one of them. Blow up at the first singularity three times, the second singularity once, the remaining $6$ double points, and $7-n$ generic points. The multiplicity sequence $[4,2,2,2]$ has been reduced to $[2]$, which is of type $T_{2,3}$, and that of $[3,3]$ to $[3]$, of type $T_{3,4}$. The resulting PL sphere is in the homology class $8h-4e_1-3e_2-2e_3-\cdots -2e_{10}-e_{11}-\cdots -e_{17-n}\in \text{I}_{1,17-n}$ and has singularities of types $T_{2,3}$ and $T_{3,4}$, from which the result follows for $K=T_{2,3}\# T_{3,4}$.

From~\cite[Theorem~1.1]{fenske} Case 5 with $a=d=2$, there also exists a rational septic in $\cp$ with two cuspidal singularities, one with multiplicity sequence $[4,2,2]$, and the other with $[3,3,2]$. Repeating the above construction on this curve yields a PL sphere in the same homology class having one singularity with multiplicity sequence $[3,2]$, which has type $T_{3,5}$. This proves the result for $K=T_{3,5}$.

Next, by~\cite[Theorem~1.1]{fenske} Case 1 with $a=d=2$ and $b=1$, there is a rational sextic $C$ in $\cp$ with two cuspidal singularities, having multiplicity sequences $[4,2,2,2]$ and $[2]$. Add two lines $\ell,\ell'$ to $C$ that intersect at a point of $C$, forming a triple point, and otherwise intersect $C$ in a total of $10$ double points. Smooth a double point intersecting $\ell$, and another intersecting $\ell'$. Then blow up the remaining 8 double points and the triple point. Finally, blow up once at the first singularity. The multiplicity sequences of the singularities are now $[2,2,2]$ and $[2]$, which are of type $T_{2,7}$ and $T_{2,3}$, respectively. Our PL sphere is in the required homology class, and the result follows for the case of $T_{2,3}\#T_{2,7}$ after blowing up at $7-n$ generic points.

Similarly, by~\cite[Theorem~1.1]{fenske} Case 1 with $a=d=2$ and $b=0$, there is a sextic with two singularities, with multiplicity sequences $[4,2,2]$ and $[2,2]$; and by~\cite[Theorem~1.1]{fenske} Case 8 with $a=4$, a sextic with two singularities with multiplicity sequences $[4]$ and $[2,2,2,2]$. Repeating the construction from the previous paragraph for these two sextics proves the result for $K=2T_{2,5}$ and $K=T_{2,9}$, respectively.

Next, consider a quintic $C$ with four cusps $p_1,\ldots, p_4$, with $p_1$ of type $T_{2,7}$, and the other three of type $T_{2,3}$; see~\cite[Theorem~2.3.10]{namba} or~\cite[Section~6.1.4]{moe}.
Pick a generic point $p$ on $C$. Let $\ell_1$ be the line passing through $p$ and $p_1$, and let $\ell_2$ and $\ell_3$ be two generic lines passing through $p$. Smooth out one generic point of intersection of $\ell_i$ with $C$ for $i=1,2,3$, and then blow up at $p$, once at $p_1$, and at the other intersection points of $\ell_i$ with $C$. (There are three such points on each of $\ell_2$ and $\ell_3$, and one on $\ell_1$.)
In total, we have blown up at one quadruple point, $p$, a triple point, $p_1$, and seven double points.
By blowing up at either cusp we obtain the class $8h-4e_1-3e_2-2e_3-\dots-2e_{10}$: blowing up at $p_1$ again yields a curve with four $T_{2,3}$ cusps, while blowing up at $p_2$ yields one $T_{2,5}$ and two $T_{2,3}$ cusps. Blowing up at $7-n$ generic points shows that $\mathscr{D}_n$ fills $S_n(K)$ for $K=4 T_{2,3}$ and $K=2T_{2,3}\# T_{2,5}$.
\end{proof}

\subsection{Final remarks}\label{ss:finalrmks}

We conclude this section with some comments on the limitations of this approach.
One can try to use Lemma~\ref{lemma:double} to produce fillings of $S^3_2(K)$ and $S^3_3(K)$, where $K$ is either $T_{2,5}$ or $2T_{2,3}$, with intersection form $\Gamma_{12}\oplus\langle2\rangle\oplus\langle1\rangle$ and $\Gamma_{12}\oplus\langle3\rangle$, respectively;
while we have attempted to employ this strategy, we haven't succeeded.
We will focus on the case of $+3$-surgery and of $\Gamma_{12}\oplus\langle3\rangle$, the other case being analogous.
For one thing, the adjunction inequality shows that the homology class $8h-4e_1-2e_2-\dots-2e_{12}-e_{13}$ cannot be represented by a genus-2 surface, so it certainly cannot be represented by a PL sphere with a singularity of type $T_{2,5}$ or $2T_{2,3}$.
We have tried with several other classes, obtained from $8h-4e_1-2e_2-\dots-2e_{12}-e_{13}$ by applying reflection automorphisms of $I_{1,13}$;
some of them passed the adjunction formula test, but we were unable to produce singular complex curves with the required multiplicities at the singularities.


\vphantom{{\shortstack{1\\ 2\\ 3}}}
\begin{table}
\centering
\begin{tabular}{*{14}{>{\centering\arraybackslash}m{9.5mm}}}
 rk$\backslash$det  & 1 & 2 & 3 & 4 & 5  & 6 & 7 & 8 & 9 & 10 & 11\\
\hline\\
1 &  \textbf{--} & $A_1$  & $O_1$ & $O_1$ & $O_1$ & $O_1$ & $O_1$ & $O_1$ & $O_1$ & $O_1$ & \vspace{-.15cm}$O_1$ \vphantom{{\shortstack{1\\ 2\\ 3}}}\\[12pt]
2 &  \textbf{--} & \textbf{--}  &  $A_2$ & \textbf{--} & $A_1O_1$  &   \textbf{--}  & \framebox{$A_1O_1$} & $O_2$ & $A_1O_1$ & \textbf{--} & \vspace{.2cm}\shortstack{\framebox{$O_2$} \\ $A_1O_1$} \\[12pt]
3 &  \textbf{--} & \textbf{--}  & \textbf{--}  & $A_3$ & \textbf{--}  & \textbf{--} & $A_2O_1$ & $A_1^2O_1$ &   \textbf{--} & $A_2O_1$ & \vphantom{{\shortstack{1\\ 2\\ 3}}} \\[12pt]
4 &   \textbf{--} & \textbf{--}  & \textbf{--} & $D_4$ &  \framebox{$A_4$} & \textbf{--} & \textbf{--} & $A_3O_1$ & \framebox{$A_3O_1$} & \vphantom{{\shortstack{1\\ 2\\ 3}}} &  \\[12pt]
5 &  \textbf{--}  & \textbf{--}  & \textbf{--}  &  \framebox{$D_5$} & \textbf{--}  & $A_5$ & \textbf{--} & \framebox{$D_4O_1$} & \vphantom{{\shortstack{1\\ 2\\ 3}}}  & & \\[12pt]
6 &  \textbf{--} & \textbf{--}  & \framebox{$E_6$} & $D_6$ & \textbf{--}  & \textbf{--} &  \shortstack{\framebox{$D_5O_1$}\\ $A_6$}  & \vphantom{{\shortstack{1\\ 2\\ 3}}} &  &  & \\[12pt]
7 &  \textbf{--} & \framebox{$E_7$} & \textbf{--} & $D_7$ & $E_6O_1$  & \framebox{$D_6O_1$}   & \vphantom{{\shortstack{1\\ 2\\ 3}}}&& &  & \\[12pt]
8 &  \framebox{$E_8$} & \textbf{--}  & $E_7O_1$ & $D_8$ &  \shortstack{\framebox{$D_7O_1$} \\ $E_7O_1$}& \vphantom{{\shortstack{1\\ 2\\ 3}}} & &&& & \\[12pt]
9 & \textbf{--} & \textbf{--}  & \textbf{--} & {\shortstack{\framebox{$D_8O_1$}\\ $E_7A_1O_1$ \\ $D_9$}} & \vphantom{{\shortstack{1\\ 2\\ 3}}}  & & & & & & \\[12pt]
10 &  \textbf{--} & \textbf{--}  &  \framebox{$D_9O_1$} & \vphantom{{\shortstack{1\\ 2\\ 3}}}  & &  &   &&&  & \\[12pt]
11 & \textbf{--}   &  \framebox{$D_{10}A_1$} & \vphantom{{\shortstack{1\\ 2\\ 3}}} &  &   &  &&& &  & \\[12pt]
12 & \framebox{$D_{12}$} & \vphantom{{\shortstack{1\\ 2\\ 3}}}  &   &  &  &  &  &&& &
\end{tabular}

\caption{Reduced indecomposable positive definite integral lattices of low rank and determinant taken from \cite{cs-low}. A dash ``\textbf{--}'' indicates that there are no lattices of the associated rank and determinant. The 
lattices $\mathscr{T}_n$ and $\mathscr{C}_n$ have been boxed.}\label{fig:tableroots}
\end{table}

\bibliography{references}
\bibliographystyle{alpha}

\vspace{.4cm}

\end{document}